\documentclass[12pt, reqno]{amsart}
\usepackage{amsmath}
\usepackage{color}

\newtheorem{theor}{\indent\sc Theorem}[section]
\newtheorem{corol}[theor]{\indent\sc Corollary}
\newtheorem{lemma}[theor]{\indent\sc Lemma}   
\newtheorem{prop}[theor]{\indent\sc Proposition}

\newtheorem{dfn}[theor]{\indent\sc Definition}

\newtheorem{example}{\indent\sc {Example}}

\newcommand{\real }{{\mathbb R}}

\newcommand{\lpr}[2]{\langle{#1},{#2}\rangle}

\newcommand{\mink}{{\real}_{1}^{4}}

\newcommand{\spac}{{\real}_{2}^{4}}

\newcommand{\euclidean}{{\mathbb{E}}^{3}}

\newcommand {\B}{{\mathcal B}}

\newcommand {\E}{{\mathcal E}}
\newcommand {\C}{{\mathcal C}}
\newcommand {\F}{{\mathcal F}}

\begin{document} 
\title{Geodesic Lines in Fields of Velocity}
\author{M. P. Dussan \; and \; A. P. Franco Filho} 
\footnote{Corresponding author: \  dussan@ime.usp.br (M.P. Dussan), apadua@ime.usp.br (A.P. Franco Filho)}

\maketitle

\centerline{\small{ Universidade de S\~ao Paulo, Departamento de
Matem\'atica - IME}}
\centerline{\small{CEP: 05508-090. S\~ao Paulo. Brazil}}


\begin{abstract} 
This work is a purely syntactic geometric exploration of some few elements, which are our axioms, 
that in last instance it is the set of differential equations whose  solutions give the geodesic lines of the 
Schwarzschild spacetime. We observe that non new physics principles or postulates will be introduced in this work. 
We only link the Bohr's atoms model with the Einstein's relativity through of a common geometric syntax. 
To obtain this common syntax, we will define the {\it extended Lorentz group}, which is defined to preserve the 
volume form of the Minkowski spacetime. The Schwarzschild spacetime will be defined as a manifold associated to a 
set of radial fields of velocities within of the four-dimensional Minkowski vectorially space form. 
Our procedure includes a comparison of the Newtonian and the Schwarzschild times along geodesic lines. 
Our constructions have strong influence of the Einstein paper about the energy content produced by fields, 
as well as by the Schr\"odinger digression about the annihilation of matter. We can say that the first is a mechanic 
or pointwise form of the mass-energy and the second is a global form or thermodynamic of mass-energy.  
We define the orbital associated to the Kepler's laws as a set of elliptical orbits, which have equal eccentricity and equal 
major semi-axis.  Then identifying the eccentricity with the relativistic velocity we will obtain a thermodynamic 
equivalence between the increasing of mass in kinetic form in special relativity theory and an adiabatic process 
with degree of freedom equal to 2. The eccentricity will be the needed velocity  to move the revolution ellipsoid and so to obtain 
a contraction of its major axis such that it converts into a sphere with radius given by the minor semi-axis. 
Therefore we can associate to the each class of equal eccentricity orbital an unique timelike unit vector, 
which is called {\it the observer} of class. 
\end{abstract}
\maketitle 

\vspace{0.2cm}

\vspace{0.1cm}

\section{Introduction and Some Constructions} 

Throughout the text, we will cite the bibliography to make clear the origin of the propositions that we used. 
Basic relativistic results can be found in \cite{O'N} and \cite{LEM}, and the Bohr's postulates are found in the reference
\cite{P}. In \cite{DFS} a set of propositions about spacelike surfaces can be found, in particular one knows from \cite{DFS} 
that totally umbilical spacelike surfaces can be characterized as sections of an affine spacelike hyperplane and of 
the lightcone of $\mink$. On \cite{Schr} the reader can find a set of results about application of statistical thermodynamic 
to a set of material systems or atoms.  

When the result or axiom is already part of our folklore we dispense bibliographic citations. 
For instance Kepler's Laws are found in many books relevant to the subject and so on. Because Kepler's Laws will be one of 
principal assumptions in this work, we will begin our introduction studying these  laws. We reaffirm that we are interested 
in the syntax of these laws and  their realist meaning, for this work, is not relevant. Furthermore, we will assume that the 
three Kepler's laws are universal and their validity does not depend on who observes or measures them.

\vspace{0.1cm}
\subsection{First Kepler's Law} For a planet in our Sun system this law says: 
\ \lq \lq {\it The orbit of each planet is an elliptical (plane) trajectory, where the Sun has one of focal positions}." 

\subsection{Second Kepler's Law and consequences} For a planet in our Sun system this law says: 
\ \lq \lq {\it In a planetary elliptical motion the line segment from the focal Sun position until 
the planet, sweeps equal areas in equal times."}

\vspace{0.1cm}
A simple consequence of this law is that circular orbits has constant speed. This suggests the following statement:

\lq \lq When, the trajectory of a planet is a circumference (as was assumed by Copernicus), the planet moves with constant speed". 

\vspace{0.2cm}
Now, since by definition, velocity is the ratio of the arc length of between two points of trajectory by the time spent 
to run this arc, we need to have a machine to measure the time. 
{\it This machine is given by central angles in circumferences}, as we will construct in the next.

\subsection{A time machine} Our clock will be a small circumference with a point which moves 
with constant speed and with an angle division such that $\Delta \theta$ corresponds to a $1 second$ when 
$\Delta \theta = \frac{2\pi}{3600}$,  and one day is equal to $ 86400 \Delta \theta.$

\vspace{0.1cm}
Now, taking the equation of an ellipse  $\E$ in an appropriated coordinate systems $(\hat{M},x,y)$, we have 
$$\left(\frac{x}{a}\right)^{2} + \left(\frac{y}{b}\right)^{2} = 1,$$ 
where we are assuming $a \geq b > 0$ and $c = \sqrt{a^{2} + b^{2}}$.  The points $A' = (-a,0)$ and $A = (a,0)$ 
are usually called by {\it aphelion} and {\it perihelion} of the planet, and the focus are 
$F' = (-c,0)$ and $F = (c,0)$.  The other two notable points are $B' = (0,-b)$ and $B = (0,b)$. 
The origin $\hat{M}$ is the middle point of segments $F'F$, $A'A$ and $B'B$. 

\vspace{0.1cm}
{\it We will assume that the solar position is $F'$ in all this section.}

\vspace{0.2cm}

The equation in the following lemma is a syntax of the second Kepler's Law. 
 Essentially the acceleration vectors along the trajectory are pointing to origin of coordinate system. 
The origin of coordinate system can be either the center of ellipse $O$ or the solar position $F'$. The syntax in both cases 
is a common syntax, that is an unique differential equation obtained by the condition that the acceleration vector, 
at each point of the trajectory is pointing to center of coordinate system.

\vspace{0.1cm}
\begin{lemma}\label{1}
Let $\gamma(t) = r(t) (\cos \theta(t), \sin \theta(t))$ be a parametric smooth curve of the Euclidean plane $\real^{2}$. 
Its acceleration vector $\gamma''(t)$ points to origin $O$ if, and only if, 
$$r^{2} \frac{d \theta}{dt} = L \; \; \mbox{ for a real number } \; \; L \in \real.$$ 
\end{lemma}
\begin{proof}
From $\gamma' = r'(\cos \theta, \sin \theta) + r \theta' (-\sin \theta,\cos \theta)$ it follows that 
$$\gamma'' = (r'' - r (\theta')^{2})(\cos \theta, \sin \theta) + (r \theta'' + 2 r' \theta')(-\sin \theta,\cos \theta).$$ 
Therefore, $\gamma''(t)$ is parallel to $\gamma(t)$ if and only if $(r^{2} \theta')' = 0.$ 
\end{proof} 
\vspace{0.1cm}
\begin{corol}
Assuming that $\beta(t) = \rho(t) (\cos \vartheta(t),\sin \vartheta(t))$ parametrizes an elliptical orbit with center $M$ with
polar coordinates centered in its focus $F'$, it follows that

\vspace{0.1cm}
\begin{enumerate} 

\item If the time interval $\Delta t$ is equal when measured by  two equal clocks, one centered in $F'$ and the 
other in $M$, the correspondent elliptical arc measures, $s$ and $s'$, corresponding to each clock, 
are different each other. 

\item The acceleration $\beta''(t)$ is parallel to the vector position $\beta(t) - F'$ if, and only if, for constant time intervals 
$\Delta t = s$, one has 
$$\frac{\partial}{\partial t} \int_{t}^{t + s} \rho^{2} \vartheta' dt = 0 \; \; \mbox{ if, and only if, } \; \; 
\rho^{2} \vartheta' = L_{F'} \in \real.$$ 
\end{enumerate}
Therefore, the first and the second Kepler's Law hold, if, and only if, the statement (2) holds. 
Equal angles in equal times sweep equal areas.  
\end{corol} 

\subsection{The third Kepler's Law} Kepler's third Law says: {\it \lq \lq The ratio of the squares of the periods of 
the orbits of two planets is equal to the ratio of the cubes of the segments connecting the Sun and the respective planets."} 
 
\vspace{0.2cm}
The area of $\E$ does not depend of the internal segment linking an internal point $P$ and the point $\gamma(t)$. 
Moreover, in an one period the line $[P,\gamma(t)]$ sweeps the area of $\E$. 

Since the average of the pointwise speed in one cycle, is the ratio of length $l(\E)$ of $\E$ by the period $T$, from 
Calculus we have 
$$l(\E) = a \int_{0}^{2\pi} \sqrt{1 - e^{2}\cos^{2} \theta}  d \theta,$$ 
where we are taking $\gamma(t) = (a \cos \theta(t), b\sin \theta(t))$ centered in the center of the elliptical motion. 
Then, by the mean value Theorem for Riemann integral, there exists a circumference with radius $\overline{r} \in [b,a]$ 
with equal length of $\E$. Since the time is given by our clock machine, which is the angle $\theta$, we obtain for 
each circumference in question that  $l(\E) = 2 \pi \overline{r}$, 
where the average speed is a constant for each circumference. Therefore it follows that
\begin{equation}
\left(\frac{T_{1}}{T_{2}}\right)^{2} = \left(\frac{R_{1}}{R_{2}}\right)^{3} \; \; \ \ \mbox{ or } \; \; \ \ 
\frac{(T_{1})^{2}}{(R_{1})^{3}} = \frac{(T_{2})^{2}}{(R_{2})^{3}} = \mbox{constant}.
\end{equation} 

\vspace{0.2cm}

Now, if $V_{orb}$ is the average speed of an orbit, we can write for a planet with circular orbit the equation  
$$
V_{orb} = \sqrt{\frac{\alpha}{r}},
$$

\vspace{0.2cm}
that by Newton's gravitational theory $\alpha = MG$, where $M$ is the mass of Sun and $G$ is the gravitational constant. 
Moreover, the relation between the  escape speed and the orbital speed in a given point of distance $r$ of the Sun, is given by 
the Newtonian mechanic energy conservation in each stellar system, as being given by
\begin{equation}
V_{{escape}} = \sqrt{2} \; V_{{orbit}}. 
\end{equation} 

\vspace{0.1cm}

\subsection{Adiabatic process and relativistic contraction of distances} We observe that this subsection has a great 
influence of Digression topic about annihilation of matter described by Schr\"odinger in \cite{Schr} page 62.

\vspace{0.2cm}
Let $a$ and $b$ the lengths of the semi-axis of $\E$, $0 < b < a$, and let $e$ be its eccentricity, defined by  
$$e = \sqrt{1 - \left(\frac{b}{a}\right)^{2}} = \tanh \phi.$$ 
Let $\E$ be the revolution ellipsoid of a set of 
Kepler orbits associated to a relativistic speed $e = v/c$ with direction of the major axis of $\E$ such that the relativistic 
contraction of distances gives a sphere $\overline{\E}$.
 
We will assume that the increasing of mass in a particle obeys the Einstein definition of kinetic energy  
$mc^{2} =  \frac{m_{0} c^{2}}{\sqrt{1 - \beta^{2}}}$, and this increasing of energy (dues a relative motion between  
two inertial referentials), has relation in thermodynamics adiabatic processes with the Planck-Einstein energy content, 
$k T = mc^{2} = h \nu$, where $k$ the Boltzmann constant of perfect gas, $\nu$ is the linear frequency and $h$ 
is the Planck constant (\cite{P}).

\vspace{0.2cm}

Now, we will assume that the contraction of $\E$ to obtain $\overline{\E}$, is an adiabatic process of an insulated 
perfect gas, non heat is exchanged with the environment.  Our Sun moves within our Galaxy which it is moving in 
the Universe, the Cosmological Environment.

\vspace{0.2cm}

Let $P V^{\lambda} = P_{0} V_{0}^{\lambda}$ be an adiabatic process  where $\lambda$ is the adiabatic coefficient and 
$(P,V)$ and $(P_{0},V_{0})$ are the respectives pressure and volume in $\overline{\E}$ and $ \E$.  Hence, since 
$V_{0} = (4 \pi a b^{2})/3$ and $V = (4 \pi b^{3})/3$, from the relations $PV = nRT$ and 
$TV^{\lambda - 1} = T_{0} V_{0}^{\lambda - 1}$, it follows that  
$$V = V_{0} \frac{b}{a} = V_{0}\sqrt{\frac{a^{2} - c^{2}}{a^{2}}} = V_{0}\sqrt{1 - e^{2}} = \frac{V_{0}}{\cosh \phi} 
\; \; \  \mbox{ and } \; \; \ T = \frac{T_{0}}{(\sqrt{1 - e^{2}\;})^{\lambda -1}}.$$ 
Therefore from $m c^{2} = k T$ it follows that $m = m_{0} \cosh^{\lambda - 1} \phi$.  Then 
$m = m_{0} \cosh \phi$ holds if, and only if, the adiabatic constant is $\lambda = 2$ if, and only if, 
its degree of freedom is $2$.

\vspace{0.1cm}
We note that usually the degree of freedom for adiabatic processes in a mono-atomic perfect gas is $f = 3$, which implies 
$\lambda = (f + 2)/f = 5/3$. In our case, the fact $\lambda = f = 2$ means that in direction of motion we 
can have the longitudinal mass by the relativistic addition of velocity, and that in transverse direction the molecules 
preserve its degree of freedom. We have lost one of the freedom  direction in the 
decomposition of average speed in Maxwell kinetic theory of perfect gas, $\vec{v} = (v_{1},v_{2},v_{3})$. 

\vspace{0.2cm}
{\it The adiabatic (global) increasing of the energy content can be identified with relativistic (local) increasing of mass, 
where the value $m/m_{0}$ does not depend of the axis of the orbital ellipsoidal set $\E$ of Kepler orbits. It depends only 
of the eccentricity of $\E$}. 

\vspace{0.3cm}
From here, this paper is divided in three sections. The first (Section 2) is dedicated to the Minkowski 
vectorially space form $\mathbb{R}^{4}_{1}$, that using its Lorentzian coordinates, we then give an axiomatic construction of the 
Schwarzschild spacetime and we construct special totally umbilic spacelike surfaces in $\mathbb{R}^{4}_{1}$ which result from the 
intersection between of lightcone with an spacelike affine hyperplane orthogonal to an observer. In addition, 
for moving us in different Minkowski referentials in such way that the volume form and the future directed are preserved we 
establish the Generalized Lorentz group. Moreover, as applications we obtain some results about the $Z$-atom of Bohr.

In Section 3 we focus in studying the four dimensional Lorentz Manifolds $(M,\mathbf{g})$ and the geodesic lines. 
In particular we study the radial geodesic lines parametrized by a proper parameter and  consider  two types of them, namely, 
the null and the timelike radial geodesics.  
The second part of this section is focused to the non-radial geodesic lines, included the equatorial geodesics. 
In both of cases of radial or non-radial geodesics we look for solutions and by relations between the Schwarzschild 
and Newton times as well as the orbital velocities in that two times.

Section 4 is devoted for the construction of a new model of parametric Lorentzian spacetime $(M',\mathbf{g})$, 
which satisfies the Einstein field equation, and which is not flat and has Ricci tensor non-null. 
We calcule then the Einstein gravitational tensor. We also study the geodesic equations and we look for solutions. 
As long as we go, we prove  in this new model that the existence of circular geodesic motions are not possible.

\section{The Minkowski vectorially space form $\mink$} 
Let $\mink$ be the Minkowski four-dimensional vector space. It is a topological locally compact four-dimensional real 
vector space equipped with the symmetric bilinear form given by: 
$$ds^{2} = -(dx^{0})^{2} + (dx^{1})^{2} +(dx^{2})^{2} +(dx^{3})^{2}.$$

It associated to this quadratic form for $\mink$ we have the Lorentzian product  
$$\lpr{u}{v} = -u^{0} v^{0} + u^{1} v^{1} + u^{2} v^{2} + u^{3} v^{3},$$ 
relatives to the canonical base of vector space $\real^{4}$, given by the elements 
$\partial_{0} = (1,0,0,0)$, $\partial_{1} = (0,1,0,0)$, $\partial_{2} = (0,0,1,0)$ and  $\partial_{3} = (0,0,0,1).$  

\vspace{0.2cm}
A {\it Minkowski frame} or  referential is a subset of four unit vectors $\B = \{e_{0},e_{1},e_{2},e_{3}\} \subset \mink$ such that: 

(1) $\lpr{e_{i}}{e_{j}} = 0$ whenever $i \neq j$. 

(2) $\lpr{e_{0}}{e_{0}} = -1$ and $\lpr{e_{i}}{e_{i}} = 1$ whenever $i = 1,2,3$. 

(3) $\lpr{e_{0}}{\partial_{0}} = - e_{0}^{0} < 0$. 

In this case, we say that the vector $e_{0}$ is the future directed timelike vector 
of the frame $\B$ of $\mink$. 

\vspace{0.2cm}
For each frame $\B$ we define a co-frame $\B^{\star} = \{dy^{0},dy^{1},dy^{2},dy^{3}\}$ such that 
$$dy^{i}(e_{j}) = \lpr{e_{i}}{e_{j}} \; \; \mbox{ for each } \; \; dy^{i} \in \B^{\star} \; \; \mbox{ and } \; e_{j} \in \B,$$ 
and it defines the orientation of the Minkowski space taking the $4$-volume form:
$$\Omega(\mink) = - dy^{0} \wedge dy^{1} \wedge dy^{2} \wedge dy^{3}.$$ 

Next we will define the set of transformations that preserves the Lorentz product, 
so preserves the volume form. 

For each Minkowski frame $\B = \{e_{0},e_{1},e_{2},e_{3}\}$ the unit future directed timelike vector 
$e_{0} \in \B$ will be called {\it observer} and its associated {\it rest-space} is given by
$$\euclidean(e_{0}) = [e_{1},e_{2},e_{3}] = \{u^{1} e_{1} + u^{2} e_{2} + u^{3}e_{3} : (u^{1},u^{2},u^{3}) \in \real^{3}\}.$$ 

\begin{dfn}
A Lorentz transformation is a linear transformation $L : \mink \to \mink$ which preserves Minkowski frames. That means
$L(\B) = \{L(e_{0}),L(e_{1}),L(e_{2}),L(e_{3})\}$  is a Minkowski frame with  $\lpr{L(e_{0})}{L(e_{0})} = -1,$ for each given 
Minkowski frame $\B$.   
\end{dfn}

\vspace{0.1cm}
\subsection{An Axiomatic Construction of the Schwarzschild Spacetime}

We will define the set of Lorentz transformations and later we will extend those transformations to a more general associated group 
which preserves  the $4$-volume form of $\mink$. We will call that group of {\it extended Lorentz group}. According the Minkowski 
paper in \cite{LEM}, it can be considered as an axiomatization of the Special Relativity Theory of Einstein.

\vspace{0.3cm}
Let $(x^{0},x^{1},x^{2},x^{3}) \in  K$ and $(\overline{x}^{0},\overline{x}^{1},x^{2},x^{3}) \in \overline{K}$ with relative velocity 
$\vec{v} = (0,v,0,0) \in \euclidean$ of $\overline{K}$ relative to $K$ and $-\vec{v} = (0,-v,0,0) \in \euclidean$ of $K$ 
relative to $\overline{K}$. We are assuming $0 \leq v < 1$,  where the light speed is assumed to be equal to $1$. 

\vspace{0.1cm}
The Lorentz transformation $L$ and its inverse $L^{-1}$, in an infinitesimal form, restrict to the part correspondent 
to the two first variable, can be written by

$$\left\{\begin{matrix} 
d\overline{x}^{0} = \frac{1}{\sqrt{1 - v^{2}}} dx^{0} - \frac{v}{\sqrt{1 - v^{2}}} dx^{1} \\
d\overline{x}^{1} = \frac{-v}{\sqrt{1 - v^{2}}} dx^{0} + \frac{1}{\sqrt{1 - v^{2}}} dx^{1}
\end{matrix} \right. 
\; \; \; \mbox{ and } \; \; \; 
\left\{\begin{matrix} 
dx^{0} = \frac{1}{\sqrt{1 - v^{2}}} d\overline{x}^{0} + \frac{v}{\sqrt{1 - v^{2}}} d\overline{x}^{1} \\
dx^{1} = \frac{v}{\sqrt{1 - v^{2}}} d\overline{x}^{0} + \frac{1}{\sqrt{1 - v^{2}}} d\overline{x}^{1}
\end{matrix} \right.$$

\vspace{0.1cm}
\begin{lemma}
The Lorentz transformation $L$ and its inverse $L^{-1}$ given above preserve the volume form: 
$dx^{0} \wedge dx^{1} = d\overline{x}^{0} \wedge d\overline{x}^{1}$. 

Defining the proper time relatives to each referential, taking  
$d\overline{x}^{1} = 0$ in the transformation $L$  (and $dx^{1} = 0$ in $L^{-1}$) 
to measure relativistic time called dilatation, we obtain, respectively: 
$$d\overline{x}^{0} = \sqrt{1 - v^{2}} dx^{0} \; \; \; \mbox{ and } \; \; \; dx^{0} = \sqrt{1 - v^{2}} d\overline{x}^{0}.$$
\end{lemma}  

\begin{proof} 
Writing the matrices corresponding to $L$ and $L^{-1}$, one observes that assuming $dx^{1} = 0$ in the 
$L^{-1}$-matrix, it follows that $v = - \displaystyle \frac{d\overline{x}^{1}}{d\overline{x}^{0}}$, and so one has 
$$\frac{d{x}^{0}}{d\overline{x}^{0}} = \frac{1}{\sqrt{1 - v^{2}}} + 
\frac{v}{\sqrt{1 - v^{2}}} \frac{d\overline{x}^{1}}{d\overline{x}^{0}} = \frac{1 + v(-v)}{\sqrt{1 - v^{2}}} = \sqrt{1 - v^{2}}.$$

Proceeding in a similar way to the other referential, we obtain $v = \displaystyle \frac{d{x}^{1}}{d{x}^{0}}$ and 
$$\frac{d\overline{x}^{0}}{d{x}^{0}} = \frac{1}{\sqrt{1 - v^{2}}} - \frac{v}{\sqrt{1 - v^{2}}} \frac{d{x}^{1}}{d{x}^{0}} = 
\frac{1 - v(v)}{\sqrt{1 - v^{2}}} = \sqrt{1 - v^{2}}.$$ 

The preservation of the volume form follows from $\det(L) = 1 = \det(L^{-1})$. 
\end{proof}

\subsection{In the Lightcone $\C \subset \mink$} The future directed lightcone $\C \subset \mink$ is the subset 
$$\C = \{p \in \mink: \lpr{p}{p} = 0 \; \mbox{ and } \; \; p^{0} = -\lpr{p}{\partial_{0}} > 0\}.$$ 

\vspace{0.2cm}
\begin{prop}
For each $p \in \C$ there exists an unique vector in the unit sphere, namely, $\vec{n} \in S^{2} \subset \euclidean$ given 
by $\vec{n} = (0,n_{1},n_{2},n_{3})$ where $n_{i} = \frac{1}{p^{0}} \lpr{p}{\partial_{i}}$ for $i =1,2,3.$ 
Therefore, one can obtain a parametrization of the lightcone using spherical coordinates, namely, 
$\Gamma(r,\varphi,\theta) = r L(\varphi,\theta)$ where 
$$L(\varphi,\theta) = (1, \cos \theta \sin \varphi, \sin \theta \sin \varphi, \cos \varphi).$$ 

Moreover, the induced metric over $\C$ is degenerated with $ds^{2}(\C) = r^{2}(d\varphi^{2} + \sin^{2} \varphi d\theta^{2}),$ 
and with induced tensor such that $g_{11} = 0$, $g_{22} = r^{2}$, $g_{33} = r^{2} \sin^{2} \varphi$ and $g_{ij} = 0$ if 
$i \neq j$.    
\end{prop} 

In the reference \cite{DFS}  the authors show that every totally umbilical spacelike surface in $\mink$ is a sub-surface 
of the lightcone contained into an affine hyperplane. For our purposes we will carry out that construction in the next subsection.

\subsection{A construction in $\mink$} A spacelike affine hyperplane can be defined using its normal vector in $\mink$  
as follows 
\begin{equation}
H(a,\tau) = \{(t, x, y, z) \in \mink : \lpr{(t - a, x, y, z)}{\tau} = 0\},
\end{equation} 
where $\tau$ is an unit future directed timelike vector of $\mink$. If $a > 0$ then the point $P = (a,0,0,0) \in H(a,\tau)$. 

\vspace{0.1cm}
Next, we will show that the section $S^{2}(\rho) = H(a, \tau) \cap \C$ is an Euclidean sphere. In fact, when 
$\tau = \partial_{0}$ the result is obvious, then we will assume that the pair $\{\partial_{0},\tau\}$ is a linearly independent 
subset of $\mink$. 

\begin{prop}
If $\{\partial_{0},\tau\}$ is a linearly independent subset of $\mink$, then, there exists a Minkowski referential 
$\{\partial_{0}, e_{1}, e_{2}, e_{3}\}$ such that $\lpr{\tau}{e_{1}} = 0 = \lpr{\tau}{e_{2}}$, with 
$e_{3}$ belonging to  $ \mbox{Span}\{\partial_{0},\tau\} \cap \euclidean$ and such that
\begin{equation}
\tau = \cosh \phi \; \partial_{0} + \sinh \phi \; e_{3} \; \; \mbox{ where } \; \; \cosh \phi = -\lpr{\partial_{0}}{\tau}, 
\; \phi > 0.
\end{equation}
\end{prop}
\begin{proof} The existence of $e_{3}$ follows from 
$$\dim(\mbox{Span}\{\partial_{0},\tau\}) + \dim(\euclidean) - \dim(\mbox{Span}\{\partial_{0},\tau\} \cap \euclidean) = 4.$$
\end{proof}

\begin{corol}
The hyperplane orthogonal to the observer $\tau$, namely $H(\tau)$, has an orthonormal base given by 
$\{e_{1},e_{2},\vec{w}\}$ where 
$$\vec{w} = \sinh \phi \; \partial_{0} + \cosh \phi \; e_{3}.$$
\end{corol}
\begin{proof}
Since $e_{3} \in \euclidean$ is an unit vector the set $\{\tau, \vec{w}\}$ forms an orthonormal base of 
$\mbox{Span}\{\partial_{0},\tau\}$ provided that 
$\lpr{ \tau}{\vec w} = \lpr{\cosh \phi \; \partial_{0} + \sinh \phi \; e_{3}}{\sinh \phi \; \partial_{0} + \cosh \phi \; e_{3}} = 0$, 
$\lpr{\tau}{ \tau} = -1$ and $\lpr{\vec w}{\vec w} = 1$. 
\end{proof}

\vspace{0.1cm}
\begin{lemma}
Let $P = a \partial_{0}$ be a point in future directed time-axis. If $X(\lambda) = P + \lambda \vec{w}$ is the line 
passing through $P$ with direction $\vec{w}$ then, it cuts the lightcone in two points $A_{1}$ and $A_{2}$, such that 
the middle point of the line segment is $M = a \cosh \phi \; \tau$.      
\end{lemma}

\begin{proof}
Since $X(\lambda) = a \partial_{0} + \lambda (\sinh \phi \; \partial_{0} + \cosh \phi \; e_{3})$, from 
$$\lpr{X(\lambda)}{X(\lambda} = -a^{2} - 2 a \lambda \sinh \phi - \lambda^{2} \sinh^{2} \phi + \lambda^{2} \cosh^{2} \phi = 0$$
it follows $(\lambda - a \sinh \phi)^{2} = a^{2} \cosh^{2} \phi$. So for this quadratic equation we have two solutions, 
namely, $\lambda_{1} = -a e^{-\phi}$ and $\lambda_{2} = a e^{\phi}$, and hence, it follows that 
$\lambda_{1} + \lambda_{2} = 2 a \sinh \phi$. Therefore, from the facts $M = X(a \sinh \phi)$ and 
$$X(a \sinh \phi) = a \cosh^{2} \phi \; \partial_{0} + a \sinh \phi \cosh \phi \; e_{3} = 
a \cosh \phi(\cosh \phi \; \partial_{0} + \sinh \phi \; e_{3})$$ 
it follows that $M = a \cosh \phi \; \tau$, $A_1 = X(\lambda_1)$ and $A_2 = X(\lambda_2)$.
\end{proof}

\vspace{0.1cm}
\begin{corol}
The sphere $S^{2}(M,\rho)$ of center $M$ and radius $\rho = a \cosh \phi$, of the hyperplane $H(a,\tau)$ passing through  
$P$ and orthogonal to $\tau$, has parametric expression 
$$X(\varphi,\theta) = M + a\cosh \phi \left[\cos \theta \sin \varphi \; e_{1} + 
\sin \theta \sin \varphi \; e_{2} + \cos \varphi \; \vec{w}\right].$$ 

Moreover, since $\lpr{X(\varphi,\theta)}{X(\varphi,\theta)} = 0$ it follows that $X(\varphi,\theta) \subset \C$.
\end{corol}

Next, let us denote by $\hat{P}$ the parallel  projection to $\partial_{0}$ (or orthogonal to $\euclidean$) of a point $P$ into 
$\euclidean$, that is $\hat{P} = P + \lpr{\partial_{0}}{P} \partial_{0}$. Then we have the following lemmas. 

\vspace{0.1cm}
\begin{lemma}
The $\partial_{0}$-parallel projection of the sphere $S^{2}(M,\rho)$ to the correspondent rest-space $\euclidean(\partial_{0})$ 
is the ellipsoid of equation
\begin{equation}
\hat{X}(\varphi,\theta) = \hat{M} + a\cosh \phi \left [\cos \theta \sin \varphi \; e_{1} + 
\sin \theta \sin \varphi \; e_{2} + \cosh \phi\; \cos \varphi \; e_{3}\right].
\end{equation}
\end{lemma}

\vspace{0.1cm}
\begin{lemma}
The set of planes passing through of the points $A_{1}$ and $A_{2}$ that are contained in the hyperplane $H(a, \tau)$, 
intercepts the lightcone in a set of circles with center $M$, which projects onto a set of ellipses, 
all having as one of the focus the origin $(0,0,0,0)$. 

Moreover, the aphelion $A_1$ and the perihelion $A_2$ are given by 
$A_{1} =  \displaystyle \frac{a(1+ e^{-2\phi})}{2}( \partial_0 - e_3)$ and $A_{2} = \displaystyle 
\frac{a(1+ e^{2\phi})}{2}(\partial_0 + e_3)$.  
\end{lemma}
\begin{proof}
The points $A_{1} = X(\lambda_{1})$ and $A_{2} = X(\lambda_{2})$ are given by  
$$A_{1} = a \partial_{0} - a e^{-\phi} (\sinh \phi \; \partial_{0} + \cosh \phi \; e_{3})  = 
a \frac{1+ e^{-2\phi}}{2}( \partial_0 - e_3),$$ 
$$A_{2} = a \partial_{0} + a e^{\phi} (\sinh \phi \; \partial_{0} + \cosh \phi \; e_{3}) = 
a \frac{1 + e^{2\phi}}{2}(\partial_0 + e_3).$$

\vspace{0.1cm}
Now, each one of those planes has a spacelike normal $n \in \mbox{Span}\{e_{1},e_{2},\vec{w}\}$ with $\lpr{\vec{w}}{n} = 0$, 
because by assumption $\vec{w}$ is a direction of all those planes. Hence   $\lpr{\partial_{0}}{n} = 0$. Thus the set of 
normal $n$ is given by $n(\xi) = \cos \xi \; e_{1} + \sin \xi \; e_{2}.$ 
\end{proof}

Next we can take the unit  future directed timelike vector $\tau$ and the unit spacelike vector $\vec{w}$ to construct a 
Minkowski referential, one for each planet, such that the elliptical motion in the correspondent rest-space 
$\euclidean(\partial_{0})$ is a section of the plane $\mbox{Span}[\tau, \nu]$ and of the lightcone 
$\C \subset \mink$. So, a circular motion in the rest-space $\euclidean(\tau)$ with the Sun in central position. 

We note that the projection of $\vec{w}$ into $\euclidean(\partial_{0})$, corresponds to the direction support of the 
aphelion and the perihelion of the elliptical planetary motion. 

\vspace{0.2cm}

In the following we construct other Minkowski referential associated to a spacelike plane $V$ and its orthogonal 
complement $V^{\perp}$, where the timelike vector $\tau \in V^{\perp}$ has minimal $\tau^{0}$ relative to any 
other observer $V^{\perp}$. 

\begin{lemma}\label{03}
Given a spacelike plane $V = \mbox{Span}\{v_{1},v_{2}\} \subset \mink$ with $\{v_1, v_2\}$ orthonormal, 
the unique orthonormal complement $T = V^{\perp}$ is a timelike plane  which has a basis $\{\tau,\nu\}$ such that

\begin{enumerate}
\item $\{\tau, v_{1},v_{2}, \nu\}$ is a Minkowski referential. 

\item The future directed timelike vector $\tau$ is given by 
$$\tau = \frac{1}{\sqrt{1 + (v_{1}^{0})^{2} + (v_{2}^{0})^{2}}} (\partial_{0} + v_{1}^{0} \; v_{1} + v_{2}^{0} \; v_{2}),$$ 
where $v_{i}^{0} = -\lpr{v_{i}}{\partial_{0}}$,  $i = 1,2.$ 

\item $\lpr{\nu}{\partial_{0}} = -\nu^{0} = 0$. 
\end{enumerate}

Therefore, if $\{t,n\}$ is an other positive oriented orthonormal basis for $T$ then $\tau^{0} \leq t^{0}$, and the equality 
holds if, and only if, \ $t = \tau$.  
\end{lemma}
\begin{proof}
From definition of $\tau$ in statement (2), it follows that $\lpr{\tau}{v_{i}} = \lpr{v_{i}}{\partial_{0}} + v_{i}^{0} = 0$, 
then in particular we have $\{\tau, v_{1}, v_{2}\}$ is an orthonormal set of vectors. Moreover, since 

$\lpr{(\partial_{0} + v_{1}^{0} \; v_{1} + v_{2}^{0} \; v_{2})}{(\partial_{0} + v_{1}^{0} \; v_{1} + v_{2}^{0} \; v_{2})} = 
-1 - 2(v_{1}^{0})^{2} + (v_{1}^{0})^{2} - 2(v_{2}^{0})^{2} + (v_{2}^{0})^{2} = - (1 + (v_{1}^{0})^{2} + (v_{2}^{0})^{2}),$ 
one has  $\lpr{\tau}{\tau} = -1$. 

The condition (3) follows from taking the unique unit vector $\nu \in T \cap \euclidean$ 
such that $\Omega(\mink)(\tau, v_{1},v_{2}, \nu) = -1.$  
\end{proof} 
\vspace{0.1cm}

\begin{corol}
The set of Lorentz transformations that preserves the pair $(V,V^{\perp})$, is given by a trigonometric angle function $\theta$ 
and a hyperbolic angle function $\phi$, such that in matrix notation, results in
\begin{equation}
\left[\begin{matrix} e_{1}(\theta) \\ e_{2}(\theta) \end{matrix}\right] = 
\left[\begin{matrix} \cos \theta & - \sin \theta \\ -\sin \theta & \cos \theta \end{matrix}\right] \; 
\left[\begin{matrix} e_{1} \\ e_{2} \end{matrix}\right] \; \; \mbox{ and } \; \; 
\left[\begin{matrix} t(\phi) \\ n(\phi) \end{matrix}\right] = 
\left[\begin{matrix} \cosh \phi & \sinh \phi \\ \sinh \phi & \cosh \phi \end{matrix}\right] \; 
\left[\begin{matrix} \tau \\ \nu \end{matrix}\right].
\end{equation}
\end{corol}

\vspace{0.2cm}
We note that a relationship between those two constructions of the Minkowski referential tells us about the direction of 
velocity $\vec{w}$ defined in  Proposition 2.4 and the vector $\nu$ defined in Lemma \ref{03}. In fact, 
because $\nu$ is orthogonal to both vectors $\tau$ and $\partial_{0}$, it follows that $\lpr{\nu}{\vec{w}} = 0$.  

\begin{corol}
If the basis given by Proposition 2.4 is $\{\partial_{0}, e_{1}, e_{2}, e_{3}\}$ and the basis given by 
Lemma \ref{03} is $\{\tau, v_{1},v_{2}, \nu\}$, then, there exists a trigonometric angle $\theta(\nu)$ such that 
\ $\nu = \cos \theta(\nu) \; e_{1} - \sin \theta(\nu) \; e_{2}$, 
and for $\vec{w}$ as defined in Proposition 2.4, the hyperbolic angle $\phi$ is such that
\begin{equation}
\lpr{\vec{w}}{e_{3}} = \cosh \phi = - \lpr{\partial_{0}}{\tau} = \tau^{0}.
\end{equation} 
This last equation shows that the adiabatic process and the relativistic contraction of distances in the 
subsection 1.5, corresponds to an observer $\tau$ that relative to the observer $\partial_{0}$, has minimal kinetic energy 
$\tau^{0}$, and the eccentricity of ellipsoid $\E$ is given by $e = \tanh \phi$ where the hyperbolic angle is given by 
$\phi = \ln(\tau^{0} + \sqrt{(\tau^{0})^{2} - 1})$.
\end{corol}

\vspace{0.1cm}
\subsection{Generalized Lorentz group} Let $\vec v$ be a field of velocity 
$\vec{v} : U \subset \real^{3} \longrightarrow TU \equiv \real^{3},$ 
where $U$ is an open subspace of the topology of Euclidean vector space $\real^{3}$. For each $p \in U$, we assume  
$\vec{v}(p) = v(p) \hat{\partial}_{v}(p)$ such that
$$0 \leq v(p) \leq 1, \; \; \mbox{ light speed } \; \; c = 1 \; \; {\rm and} \ \  
\hat{\partial}_{v}(p) = \frac{1}{v(p)} \vec{v}(p) \ \ {\rm when} \ \ \vec{v}(p) \neq 0.$$ 

Now, Let $\{\hat{\partial}_{v}(p), \hat{\partial}_{2}(p), \hat{\partial}_{3}(p)\}$ be an inertial (ou Galilei) 
referential for $T_{p}U$ associated to each field of velocity, where we are assuming that it is a pointwise set of 
orthonormal vectors at each point $p \in U$.  Then the equations  
\begin{equation}
V(p) = \partial_{0} + v(p) \; \hat{\partial}_{v}(p) \; \; \mbox{ and } \; \; 
N(p) = \frac{-1}{\lpr{V(p)}{V(p)}} (v(p) \partial_{0} + \hat{\partial}_{v}(p)) 
\end{equation} 
define the lifting of $\vec v$ to $\mink$, namely, $V : \real_{1}^{1} \times U \longrightarrow T\real_{1}^{1} 
\times U \subset \mink$, and its orthogonal field $N$, respectively. Then we have the following result.

\begin{prop}
For each $p \in U$ the lifting of the fields of velocity $\vec{v} \neq 0$ are such that
$$\lpr{V(p)}{V(p)} = g_{00}(p) = -1 + v^{2}(p) \; \; \mbox{ and } \; \; \lpr{N(p)}{N(p)} = g_{11}(p) = \frac{1}{1 - v^{2}(p)},$$ 
and by definition $g_{01} = \lpr{V(p)}{N(p)} = 0$. Moreover 
$$V(p) \otimes N(p) - N(p) \otimes V(p) = \partial_{0} \otimes \hat{\partial}_{v}(p) -  \hat{\partial}_{v}(p) \otimes \partial_{0},$$ 
that means that the volume form of $\mink$ is preserved by $\{V(p), N(p)\}$, or equivalently, the set 
$\{{\partial}_{0}(p), \hat{\partial}_{v}(p), \hat{\partial}_{2}(p), \hat{\partial}_{3}(p)\}$ is a Minkowski referential 
and the lifting of the inertial referential $\{\hat{\partial}_{v}(p), \hat{\partial}_{2}(p), \hat{\partial}_{3}(p)\}$ preserves 
the volume form of $\mink$. 
\end{prop}

Next we are interested in obtaining the Schwarzschild metric tensor associated to radial fields of velocity 
$v : \euclidean(\partial_{0}) \to \euclidean(\partial_{0})$ such that $\vert v(p) \vert < c$. Here the radial term
means that $(0,0,0,0) \in \{l \in \euclidean(\partial_{0}) : l= p + \lambda v(p), \ \lambda \in \mathbb{R}\}$.
Now, we are ready to define the extended Lorentz group as follows.
\begin{dfn}
The extended Lorentz group is a set of linear transformations $L$ of $\mink$ that preserves the volume form and 
the future directed. That means the set of $L$ with
$$ \det[L] = 1 \; \; \; \mbox{ and } \;\; \; \lpr{L(\partial_{0})}{L(\partial_{0})} < 0.$$ 
For each $L$ that belongs to the extended Lorentz group the proper time associated to $L$ is the $1$-form given by 
$$d\tau = \sqrt{- \lpr{L(\partial_{0})}{L(\partial_{0})}} \; dt,$$ 
and its line element, relatives to the Minkowski frame $\{\partial_{i}\}_{i = 0}^{3}$ and to
$\{L(\partial_{i})\}_{i = 0}^{3}$, is given by 
$$ds^{2}(L) = \sum_{i,j = 0}^{3} g_{ij} dx^{i} dx^{j} \; \mbox{ where } \; g_{ij} = \lpr{L(\partial_{i})}{L(\partial_{j})}.$$ 

\end{dfn}

\vspace{0.2cm}

\begin{example}[Newtonian gravitational fields] 
Let $\vec v$ be the radial field, defined in $U \subset \real^{3}$ and given by
$$\frac{1}{c} \vec{v}(r) = \sqrt{\frac{2MG}{r c^{2}}} \hat{\partial}_{r} \; \; \; \;
\mbox{ where } \; \; \; \; \hat{\partial}_{r} = \frac{\vec{r}}{r}
$$ 
with $r > 2MG/c^{2}$, $M > 0$ and $G$ being the Newtonian gravitational constant. 

\vspace{0.1cm}
We take the frame for spherical coordinates given by $\Gamma(r,\varphi,\theta)$ to define the lifting
$$V(r) = \partial_{0} + \sqrt{\frac{2MG}{r c^{2}}} \hat{\partial}_{r} \; \; \; \mbox{ and } \; \; \;
N(r) = \frac{1}{1 - \frac{2MG}{r c^{2}}} \left(\sqrt{\frac{2MG}{r c^{2}}} \partial_{0} + \hat{\partial}_{r}\right),$$ 
and simultaneously preserves the others two components of $\hat{\partial}_{\varphi}$ and $\hat{\partial}_{\theta}$. 
Then, we  have $\tilde M$ locally parametrized by the new coordinates 
$(\tilde t, \tilde r, \varphi, \theta)$ in which the Lorentzian tensor is given by the line element 
$$ds^{2}(\tilde{M}) = -\left(1 - \frac{2MG}{\tilde{r} c^{2}}\right) d\tilde{t}^{2} + \frac{1}{1 - \frac{2MG}{\tilde{r} c^{2}}} \; 
d\tilde{r}^{2} + \tilde{r}^{2} (d\varphi^{2} + \sin^{2} \varphi \; d\theta^{2}).$$

Now taking the new parameters $r$ and $t$ given by the transformation  
$$ r = \frac{c^{2}}{2MG} \tilde{r} \; \; \; \mbox{ and } \; \; \;
t = \frac{c^{2}}{2MG} \tilde{t},$$ 
we obtain for $r > 1$ that the metric tensor for the basic manifold $M$ is given by 
$$ds^{2}(M) = -\frac{r - 1}{r} dt^{2} + \frac{r}{r - 1} dr^{2} + r^{2} d\varphi^{2} + r^{2} \sin^{2} \varphi \; d\theta^{2}.$$ 
\end{example}

\vspace{0.2cm}
\begin{example}[Coulombian electrostatic fields]
Let $(M,+Q)$ and $(m,-q)$ be Coulombian charges $+Q$ and $-q$, with mass $M$ and $m$, respectively. 
It assumes that the distance between them is $r>0$. The Coulomb Law for this charges system defines an 
attractive force between the two charges given by 
$$\vec{F} = \frac{-qQ}{4\pi \epsilon_{0} r^{2}} \hat{\partial}_{r},$$
where the factor $\frac{1}{4 \pi \epsilon_0}$ is the Coulombian constant. 

Now taking the Newton's definition of force $\vec{F} = m \displaystyle \frac{d \vec{v}}{dt}$ in the charge $(m, -q)$, 
results in 
$$\frac{d \vec{v}}{dt} = \frac{-qQ}{4\pi \epsilon_{0} m r^{2}} \hat{\partial}_{r}.$$
It assumes  the values $Q = e = \vert -q \vert$, where $e$ is the absolute value of the charge of one electron. We obtain 
the fields of radial velocity for the electron, namely,  
$$\frac{\sqrt{2}}{2c} \vec{v} = - \sqrt{\frac{e^{2}}{4\pi \epsilon_{0} mc^{2} r}} \; \hat{\partial}_{r}.$$   
\end{example}

\vspace{0.2cm}
We note that the second example above can be used for constructing a model for an ideal hydrogen atom. Indeed, this was 
made by Bohr. Next subsection we analyze it.

\vspace{0.1cm}
\subsection{The $Z$-atom of Bohr} According reference \cite{P}, to construct his atomic theory, 
Bohr introduces the following postulates: 

\vspace{0.2cm}
1. In an atom with nucleon containing an integer number $Z$ of protons  (thus with charge $+Ze$), there exists a set of 
orbital positions $n = 1,2,3,...$ such that each electron (of charge $-e$) has linear momentum $p = m v$ given by 
$$m v = \frac{h n}{2 \pi r} \; \; \; \; \mbox{ where } \; \; \; \; h \approx 6,62 \times 10^{-34} (SI) \; \; 
\mbox{is the Planck's constant}.$$ 

2. The centripetal force $F_c$ is given by Coulomb's Law by 
$$F_{c} = \frac{m v^{2}}{r} = \frac{Z e^{2}}{4 \pi \epsilon_{0} r^{2}} \; \; \;  \mbox{ for each orbital electron}.$$  

\vspace{0.1cm}
We note that since the potential is a function of type $r \mapsto K/r$, this second postulate is 
equivalent to assume Kepler's Law for this pair of mass-charges. 

\vspace{0.2cm}
3. In a Z-atom, the electrons in one orbital position do not emit electromagnetic energy. Only when they jump from one orbital 
to another orbital position, they emit or absorb energy according the Bohr postulate given by the relationship 
$$E_{Total}(n + 1) - E_{Total}(n) = - h \nu(n, n+1),$$ 
where $E_{Total} (n)$ is the total energy in the $n$-orbital, which is the sum of the kinetic energy and Coulomb potential energy
associated to each orbital position, and $\nu (n, n+1)$ is the electromagnetic frequency emitted when the electron passes da 
$n$-orbit to the $(n+1)$-orbit. The signal in $-h \nu(n,n+1)$ means that the atom emits electromagnetic energy or radiation with 
frequency $\nu(n,n+1)$ when jump from a orbital to other orbital of less energy level. 

Since, the kinetic energy $E_{Kin}$ and the Coulombian potential energy $E_{Pot}$ are given by 
$$E_{Kin} = \frac{1}{2} mv^{2} = \frac{Z e^{2}}{8 \pi \epsilon_{0} r} \; \; \; \; \mbox{ and } \; \; \; \; 
E_{Pot} = -\frac{Z e^{2}}{4 \pi \epsilon_{0} r}$$ it follows that  $E_{Total} = - E_{Kin}$. 

\vspace{0.2cm}
Next we show the following result which is already known.
\vspace{0.2cm}
\begin{prop}
From the Bohr's postulates 1 and 2 it follows that the velocity $v$ of the electron and the product $mr$, depend of $Z$ and $n$ 
in the form  

\begin{equation}\label{19}
v(Z,n) = \frac{Z e^{2}}{2 h \epsilon_{0} n} \; \; \; \mbox{ and } \; \; \; m r(Z,n) = \frac{h^{2} \epsilon_{0}}{\pi Z e^{2}} n^{2}.
\end{equation} 
Moreover, the Bohr radius $a_0$ for the hydrogen atom (here $Z = 1 = n$) and the Sommerfeld constant $\hat \beta$ are given by   
\begin{equation}\label{20}
 m a_{0} = \frac{h^{2} \epsilon_{0}}{\pi e^{2}} \; \; \; \mbox{ and} \; \; \; 
 \hat{\beta} = \frac{\hat{v}}{c} = \frac{e^{2}}{2ch\epsilon_{0}}, \; \; \mbox{with}  \; \; \hat v = \frac{e^2}{2h \epsilon_0}, 
\end{equation}
where $a_{0} \approx 5,29 \times 10^{-11} \mbox{meters}$ and $\hat{\beta} = 1/ 137$ is 
a physics dimensionless constant.     
\end{prop}
\begin{proof}
First, we observe that $mv v = \displaystyle \frac{ v h n}{2 \pi r} = \frac{Z e^{2}}{4 \pi \epsilon_{0} r}$ implies 
$v = \displaystyle \frac{Z e^{2}}{2\epsilon_{0} h n}$. So, for $Z = 1 = n$ it follows 
$$\hat{v} = \frac{e^{2}}{2 h \epsilon_{0}} \; \; \; \mbox{ and } \; \; \; \hat{\beta} = \frac{e^{2}}{2ch\epsilon_{0}}.$$ 
Now, for getting the expression $mr$ we note that
$$m r(Z,n) = \frac{h n}{2 \pi v} = \frac{h n}{2 \pi} \; \frac{2 \epsilon_{0} h n}{Z e^{2}} = 
\frac{h^{2} \epsilon_{0}}{\pi Z e^{2}} n^{2}.$$
\end{proof}

\begin{corol}\label{27}
For each Bohr's $Z$-atom one has that 
$\beta(Z,n) = \hat{\beta} \displaystyle \frac{Z}{n}$, where from equation (9) the term  
$\beta(Z,n)= \frac{v(Z,n)}{c}$. Therefore  
$$ \beta(Z,1) = 1 \; \Longleftrightarrow \; 
Z = \frac{1}{\hat{\beta}} = 137.$$ 

Moreover, $\beta(Z,Z) = \beta(n,n) = \hat{\beta}$.  Assuming that \lq \lq all material system of bodies-energy the velocity of each 
bodies-energy is less of velocity of light $c$, we have $Z$-atoms for maximal $Z_{max} = 136 < 137$. 
The $137$-atom will be called of {\it  ideal limit atom}. 
\end{corol}

\begin{prop}
From Bohr postulates it follows that in a hydrogen atom, the jump of one electron from one orbital position to other orbital position 
emits or absorb energy given by   
\begin{equation}\label{23}
h \nu(n,n+1) = \frac{e^{2}}{8 \pi \epsilon_{0} a_{0}} \left\{\frac{1}{n^{2}} - \frac{1}{(n+1)^{2}}\right\} = 
\frac{m e^{4}}{8 h^{2} \epsilon_{0}^{2}} \left\{\frac{1}{n^{2}} - \frac{1}{(n+1)^{2}}\right\}.
\end{equation}
\end{prop} 
\begin{proof}

Using equations (\ref{20}), the definition of kinetic energy and the formula for $v(n) = c \beta(n) =  \hat{v}/n$,  
we obtain  
\begin{equation}\label{22}
E_{Kin}(n) = \frac{1}{2} m v^{2}(n) = \frac{1}{2}  \frac{m \hat{v}^{2}}{n^{2}} = \frac{m e^{4}}{8 h^{2} \epsilon_{0}^{2}} \; 
\frac{1}{n^{2}}.
\end{equation}
Now, we can replace $m$ by the Bohr radius $a_{0}$ where 
$m a_{0}  = \frac{h^{2} \epsilon_{0}}{\pi e^{2}}$, for obtaining from equation (\ref{22}) that
$$E_{Kin}(n)  = \frac{e^{2}}{8 \pi \epsilon_{0} a_{0}} \; \frac{1}{n^{2}}.$$ 
The second equality in equation (\ref{23}) follows from the last relation above.   
\end{proof}

\begin{corol}
\begin{enumerate}
\item \ The sum of the telescopic series obtained from equation (\ref{23}) establishes the energy of ionization 
for the hydrogen atom as being
\begin{equation}\label{25}
h\nu_{\infty} = \sum_{n = 1}^{\infty} h \nu(n,n+1) = \frac{e^{2}}{8 \pi \epsilon_{0} a_{0}} = \frac{m e^{4}}{8 h^{2} \epsilon_{0}^{2}}.
\end{equation} 

\item Using the relativistic relation and the Planck's relation, so that $E = m c^{2} = h \nu$,  and taking as definition 
$\Delta m c^{2} = h \nu_{\infty}$ and $h \nu_{e} = m c^{2}$, it follows  
\begin{equation}\label{26}
\Delta m = \frac{m}{2} \hat{\beta}^{2} \; \; \; \mbox{ and } \; \; \; \nu_{\infty} = \frac{\nu_{e}}{2} \hat{\beta}^{2} \; \; 
\mbox{ or } \; \; \lambda_{e} = \frac{\lambda_{\infty}}{2} \hat{\beta}^{2}.
\end{equation}
\item Using the relation for $ma_{0}$ in equation (\ref{20}), it follows that the relation between the Bohr radius 
of $H_{1}$, namely $a_{0}$, and the nucleus singularity of $H_{1}$, namely $r_0$, is given by  
$$r_{0} = \frac{e^{2}}{4 \pi \epsilon_{0} m c^{2}} = \frac{e^{2}}{4 \pi \epsilon_{0} c^{2}} \; 
\frac{\pi e^{2} a_{0}}{h^{2} \epsilon_{0}} = a_{0} \hat{\beta}^{2}.$$

\item For a $Z$-atom one has
\begin{equation}\label{28}
h \nu(Z;n,n+1) = \frac{m  Z^2 e^{4}}{8 h^{2} \epsilon_{0}^{2}} \left\{\frac{1}{n^{2}} - \frac{1}{(n+1)^{2}}\right\} = 
\frac{Z^2 e^{2}}{8 \pi \epsilon_{0} a_{0}} \left\{\frac{1}{n^{2}} - \frac{1}{(n+1)^{2}}\right\}.
\end{equation} 

\vspace{0.1cm}
\item The mass of one electron $m_{e}$ and the Bohr radius $a_{0}$ for the hydrogen atom 
can be determined by the knowledge of spectrum of $H_{1}$ using equation (\ref{23}). 

\vspace{0.1cm}
\item Moreover, using $Z = 1$ and the relation $\lambda = \displaystyle \frac{h}{mc}$, from equations 
(\ref{25}) and (\ref{26}), it follows that
\begin{equation}
\frac{hc}{\lambda_{\infty}} = \frac{e^{2}}{8 \pi \epsilon_{0} a_{0}} \; \Longleftrightarrow \; \frac{2 \pi a_{0}}{\lambda_{\infty}} = 
\frac{e^{2}}{4 h \epsilon_{0} c} = \frac{\hat{\beta}}{2}.
\end{equation}
Moreover, from equation (\ref{26}) for $\lambda_e$, it follows 
$ \lambda_{e} = 2 \pi a_{0} \hat{\beta}.$ In addition 
\begin{equation}
\hat{\beta} = \frac{2 \pi r_{0}}{\lambda_{e}} \; \; \; \; \mbox{ or } \; \; \; \; \frac{\lambda_{e}}{2 \pi r_{0}} = 137 = 
\frac{2 \pi a_{0}}{\lambda_{e}}.
\end{equation}
\end{enumerate}
\end{corol}
\begin{proof}
From Corollary \ref{27} applied to the hydrogen atom we obtain $\beta(n) = \hat{\beta}\frac{1}{n}$ and so equation (\ref{28}) 
can be written as  
$$h \nu(n,n+1) = \frac{e^{2}}{8 \pi \epsilon_{0} a_{0}} \hat{\beta}^{-2}(\beta^{2}(n) - \beta^{2}(n+1)),$$ 
and then, from $r_{0} = a_{0} \hat{\beta}^{2}$ it follows that  
$$h \nu_{\infty} = \frac{e^{2}}{8 \pi \epsilon_{0} r_{0}} \sum_{n=1}^{\infty} (\beta^{2}(n) - \beta^{2}(n+1)) = 
\frac{e^{2}}{8 \pi \epsilon_{0} r_{0}} \hat{\beta}^{2}.$$ 
Now, since $h \nu = \frac{h c}{\lambda}$ we obtain
$$\frac{1}{\lambda_{\infty}} = \frac{e^{2}}{8 \pi \epsilon_{0} h c r_{0}} \hat{\beta}^{2} = \frac{1}{4 \pi r_{0}} \hat{\beta}^{3} \; \; \; \; 
\mbox{ or } \; \; \; \; r_{0} = \frac{1}{4 \pi} \lambda_{\infty} \hat{\beta}^{3} = \frac{\lambda_{e}}{2 \pi} \hat{\beta}.$$ 
\end{proof}

\section{The Four Dimensional Lorentz Manifolds $(M,\mathbf{g})$}

\vspace{0.3cm}
In this section we focus in studying the timelike and null geodesic lines of four dimensional Lorentz Manifolds $(M,\mathbf{g})$. 
Particularly, the called {\it the radial geodesic lines} and  {\it periodic geodesic lines} which we will see that are Kepler circular orbits.   

\begin{dfn}
For $(t,r,\varphi,\theta) \in M = \real \times ]1,+\infty[ \times \real^{2}$ let $(M,\mathbf{g})$ be the 
Lorentz manifold equipped with the semi-defined symmetric covariant second-order tensor given by the line 
element
$$ds^{2}(M) = -\frac{r - 1}{r} dt^{2} + \frac{r}{r - 1} dr^{2} + r^{2} d\varphi^{2} + r^{2} \sin^{2} \varphi \; d\theta^{2}.$$

A geodesic line in $M$ is an one-parameter curve $\alpha(s) = (t(s),r(s),\varphi(s),\theta(s)) \in M$,  $s \in I$, 
where $I \subset \real$ is an open interval, satisfying
\begin{equation}\label{29}
\alpha''(s) = 0 \; \; \mbox{ for each } \; \; s \in I,
\end{equation} 
where, here, $\alpha''$ means the second order covariant derivative associated to the Levi-Civitta connection of $M$. 

We say that $s \in I$ is a proper parameter for a curve $\alpha : I \to M$ if for each $s \in I$, 
$\lpr{\alpha'(s)}{\alpha'(s)}' = 0$ and $\lpr{\alpha'(s)}{\alpha'(s)} = a \neq 0$, where $a \in \real$. 

\vspace{0.2cm}
A null curve is a non-constant curve where $\lpr{\alpha'(s)}{\alpha'(s)} = 0$.   

\vspace{0.2cm}
The subset traced by a curve $\alpha : I \to M$ will be called of line or trajectory.  A re-parametrization 
of $\alpha(s)$ is another curve $\overline{\alpha}(\eta)$ with domain $J \subset \real$ such that 
$\alpha(I) = \overline{\alpha}$(J). In other words, these two curves trace equal trajectory in $M$. 
\end{dfn}

It notes that since the metric tensor of $M$ can be seen as a sum of two metric tensor, namely, 
$ds^{2}(M) = ds^{2}(P) + ds^{2}(\C)$ with
$$ds^{2}(P) = - h dt^{2} + \frac{1}{h} dr^{2} \; \; \mbox{ and } \; \; ds^{2}(\C) = r^{2}(d\varphi^{2} + \sin^{2} \varphi \; d\theta^{2}),$$ 
where $P$ denotes  de Schwarzschild half plane $\real \times ]0,+\infty[$ and $\C$ denotes the lightcone of the $\mink$, we can see our 
spacetime $(M,\mathbf{g})$ as a sub-manifold of the manifold product $P \times \C$. That is in fact a graph given by $r = \rho$ when, 
parametrically, we take the lightcone given by the map $\Gamma(\rho, \varphi, \theta)$. See subsection 2.2 and Proposition 2.3, 
replaced $r$ by $\rho$.

Therefore, we can see the curve $\alpha(s)$ as a ordered pair $\alpha(s) = (\beta(s),\gamma(s))$ 
where $\beta(s) = (t(s),r(s))$ is the component corresponding to $P$ and $\gamma(s) = \Gamma(r(s),\varphi(s),\theta(s))$ 
corresponding to $\C$.

 If we take coordinates $((t,r),(\rho, \varphi,\theta)) \in \real^{5}$ to represent points in $P \times \C$, the graph of the 
function $\rho = r$ is isometrically equivalent to $M$. Therefore, we can say that each curve in $M$ can be seen as a curve of type 
$$\alpha(s) = F(t(s),r(s),\varphi(s), \theta(s)) \; \; \mbox{ where } \; \; 
F(t,r,\varphi,\theta) = ((t,r),(r,\varphi,\theta)) \in P \times \C.$$

\begin{lemma}{(The proper parameter lemma)}

\begin{enumerate}

\item If $\alpha : I \to M$ is a geodesic line then $s \in I$ is a proper parameter. 

\vspace{0.1cm}
\item If $\overline{\alpha} : J \to M$ is a re-parametrization such that $\overline{\alpha}''(\eta) = 0$  
then the parameters $s, \eta$ are linking each other by an affine liner transformation $\eta = p s + q$, where $p \neq 0$ and 
$p,q \in \real$.
\end{enumerate}
\end{lemma}
\begin{proof}
Item (1) is obvious from definition. For item (2), one starts by taking $\overline{\alpha}(\eta) = \alpha(s(\eta))$ and thus 
from chain rule it follows that 
$$\frac{d^{2}\overline{\alpha}(\eta)}{d\eta^{2}} = \frac{d^{2} \alpha}{ds^{2}} \left(\frac{ds}{d\eta}\right)^{2} + 
\frac{d\alpha}{ds} \frac{d^{2}s}{d\eta^{2}}.$$ 
Now, from the regularity of these two parametric curves and from the assumption that they are geodesic lines it follows 
that $s''(\eta) = 0$, which implies $\eta = p s + q$ for $p \neq 0$ and 
$p,q \in \real$.
\end{proof}

\begin{corol} Let  $\alpha : I \to M$ be a geodesic line. It assumes that the system of ordinary differential equations for $\alpha$, 
given by (\ref{29}), contains the equation $ \frac{dt(s)}{ds} = \E \frac{r(s)}{r(s) - 1}.$ 
Then, taking $s'(\eta) = 1/\E$, one can re-parametrize it to obtain the equation 
$$\frac{dt(s)}{ds} \; \frac{ds}{d\eta} = \frac{dt(\eta)}{d\eta} = \frac{r(\eta)}{r(\eta) - 1}.$$ 
\end{corol}

\begin{dfn}\label{41}
Let the effective potential for $M$ be the function $h(r) =  \displaystyle \frac{r - 1}{r}$  (such function is known as 
the Schwarzschild function). Then the parameter $\eta$ such that 
\begin{equation}
\frac{d\eta}{dt} = \frac{r - 1}{r},
\end{equation} 
 will be called the effective parameter for the geodesic lines in $M$. 

The function $\displaystyle \frac{1}{r}$, which is equal to $1 - h(r)$, will be called the 
Kepler-Newton gravitational potential. 
\end{dfn}

\begin{theor}\label{30}
 For $(M,\mathbf{g}),$ let $\eta$ be the effective parameter of a geodesic line $\alpha(\eta)$. 
Then the geodesic equation (\ref{29})  corresponds to the following ODE system: 
\begin{align}\label{31}
\frac{dt}{d \eta} = \frac{r}{r - 1}\\ 
r^{2} \sin^{2} \varphi \frac{d \theta}{d\eta} = L, \  L \in \mathbb R \ constant\\
\frac{d}{d \eta}\left(r^{2} \frac{d \varphi}{d \eta}\right) = r^{2}\sin \varphi \cos \varphi \left(\frac{d \theta}{d \eta}\right)^{2} \\
\frac{d^{2} r}{d \eta^{2}} + \frac{1}{2r(r - 1)} \left(1 - \left(\frac{d r}{d \eta}\right)^{2}\right) = 
(r -1) \left[ (\varphi')^{2} + (\theta')^{2} \sin^{2} \varphi \right]
\end{align}
\end{theor}

\vspace{0.1cm}
 We will call Equation (23) the {\it linking equation} because it is associated to the graph of the function $\rho = r$. 
Equation (20) is the essentially the definition of the effective parameter $\eta$ that works as a time that \lq \lq transports" 
 the effective potential $h(r)$ along each one geodesic lines.  The equations (21) and (22) can be seen as the necessary 
and sufficient condition to obtain a parametric curve $\hat{\gamma}(\eta)$ in $\euclidean$ such that its acceleration 
vector $\hat{\gamma}''(\eta)$ is orthogonal to each spherical surface $S^{2}((0,0,0,0)) \subset \euclidean$ at the point 
$\hat{\gamma}(\eta)$.
Next, for the proof of Theorem \ref{30} we start proving equation (20). Equations (21), (22) and (23) will be proved through 
of the development of this subsection and in according our necessities. 

\begin{proof}  (Equation (20)).  It writes $(t(\eta),r(\eta), \varphi(\eta), \theta(\eta)) = (x^{0}(\eta), x^{1}(\eta), 
x^{2}(\eta), x^{3}(\eta))$ and it denotes the Christoffel symbols by $\Gamma_{ij}^{k}$ for $i,j,k = 0,1,2,3$.

For obtaining the equation (20), we take the second order covariant derivative for the $x^0$-component,  namely, 
\begin{equation}\label{32}
\frac{d^{2} x^{0}}{d \eta^{2}} + \sum_{i,j = 0}^{3} \Gamma_{ij}^{0} \frac{d x^{i}}{d \eta} \frac{d x^{j}}{d \eta} = 0 
\; \; \; \mbox{ where } \; \; \; \Gamma_{ij}^{k} = \frac{1}{2} \sum_{m = 0}^{3} g^{km}\left\{\frac{\partial g_{im}}{\partial x^{j}} + 
\frac{\partial g_{jm}}{\partial x^{i}} - \frac{\partial g_{ij}}{\partial x^{m}}\right\},
\end{equation}
for obtaining the Christoffel  symbols $\Gamma_{ij}^{0}$. 

In fact, since, $g_{00} = - h(r)$,  $g_{22} = r^{2}$,  $g_{33} = r^{2} \sin^{2} \varphi$ and $g_{11} = \displaystyle \frac{1}{h(r)}$, 
we obtain from equation (\ref{32}) that for the non-null Christoffel symbols, $t'' + 2\Gamma_{01}^{0} t' r' = 0$. 
This latter equality implies  $\Gamma_{01}^{0} = \frac{1}{2h} \; \frac{d h}{dr}$
and so  $(h(r(\eta)) t'(\eta))' = 0$. Finally, using the effective parameter we obtain $h(r(\eta)) t'(\eta) = 1$
which shows that equation (20) holds.  
\end{proof}

In what follows we study the geodesic equations in Theorem \ref{30} for a special kind of curves which we call 
{\it radial geodesic line}. We start with its definition.

\begin{dfn} Let  $\alpha: I \to M$ with $\alpha(\eta) = t(\eta), r(\eta), \varphi(\eta), \theta(\eta))$ be a curve a geodesic line, 
equipped with effective parameter $\eta \in I$. We say that 
$\alpha$ is a radial geodesic line if and only if \
$\varphi'(\eta) = 0 = \theta'(\eta)$, for each $\eta \in I$. 
\end{dfn}

We have the following proposition.
\begin{prop}\label{42}
A radial geodesic line equipped with effective parameter $\eta$ satisfies that
\begin{equation}\label{34}
\frac{d^{2} r}{d \eta^{2}} + \frac{1}{2r(r - 1)}\left(1 - \left(\frac{d r}{d \eta}\right)^{2}\right) = 0.
\end{equation}
\end{prop}
\begin{proof} In fact, taking the geodesic equation  (\ref{29}) with effective parameter $\eta$, we have for the function $r$ that 
$$r'' + \Gamma_{00}^{1} (t')^{2} + 2\Gamma_{01}^{1} t'r' + \Gamma_{11}^{1} (r')^{2} = 0.$$ 
Since $\Gamma_{01}^{1} = 0$ we obtain 
$$\Gamma_{00}^{1} = \frac{-1}{2g_{11}} \frac{\partial g_{00}}{\partial r} = \frac{r - 1}{2r^{3}} \; \; \; \mbox{ and } \; \; \; 
\Gamma_{11}^{1} = \frac{1}{2g_{11}} \frac{\partial g_{11}}{\partial r} = \frac{-1}{2r(r - 1)}.$$ 
Now, since $t' = \displaystyle \frac{r}{r -1}$ we obtain
$$r'' +  \frac{r - 1}{2r^{3}} \; \frac{r^{2}}{(r - 1)^{2}} - \frac{1}{2r(r - 1)} (r')^{2} = 0$$ 
and then, it follows the equation (\ref{34}). It notes that it corresponds to equation (23) for when
$\varphi'(\eta) = 0 = \theta'(\eta)$, for all $\eta \in I$.  
\end{proof}

\begin{corol}\label{72}
For a radial geodesic line $\alpha$ equipped with effective parameter $\eta$, it follows that

\begin{enumerate}
\item $\lpr{\alpha'}{\alpha'} = -h(t')^{2} + \frac{1}{h} (r')^{2} = - b^{2}$, with $b \in \mathbb R$. 

\vspace{0.1cm}
\item If $b = 0$ then $\alpha$ is a null radial geodesic. Moreover, one can take $r$ as a proper parameter, namely, 
$r = \eta + k_{2}$, \ $r = k_{1} \eta + k_{2}$, with $k_1, k_2$ being real constants, and so, the solutions of geodesic 
equation (which in this case, it reduces to equation (20)) are given by 
\begin{equation}\label{35}
t(r) = r_{0} + (r + \ln(r - 1)) \; \; \;  \mbox{ or } \; \; \; t(r) = r_{0} - (r + \ln(r - 1)),
\end{equation}
where $r_{0} \in \real$.

\vspace{0.1cm}
\item If $b^{2} \neq 0$  then $\alpha$ is a timelike radial geodesic. Moreover, a general solution of the geodesic equation 
(which in this case, it reduces to equation (23)) are such that 
\begin{equation}\label{36}
\eta(r) = \pm \int \sqrt{\frac{r}{(1 - b^{2})r + b^{2}}} \; dr.
\end{equation}    
\end{enumerate}
\end{corol}
\begin{proof} Item (1) follows quickly since $\alpha$ is radial and geodesic.

Now, since $h = \displaystyle \frac{r-1}{r}$, from item (1) and from equation (\ref{31}) it follows that 
$\frac{-1}{h} + \frac{1}{h} (r')^{2} = -b^{2}$ 
which implies 
\begin{equation}\label{37}
\frac{dr}{d\eta} = \pm \sqrt{1 - b^{2} + \frac{b^{2}}{r}}.
\end{equation}

For item (2), it notes that if $b = 0$ then from equation above one gets $\frac{dr}{d\eta} = \pm 1$. 
Now, on the other hand, we have that  
$$\left <\alpha' (r), \alpha'(r) \right> = (\frac{d\eta}{dr})^2 \left <\alpha'(\eta), \alpha'(\eta) \right>.$$ 
Thus, because $\eta$ is proper parameter and $\frac{dr}{d\eta} = \pm 1$, one concludes that $r$ is also  proper parameter. 
Moreover, using again equation (\ref{31}) it follows that 
$$\frac{dt}{dr} = \pm \frac{r}{r - 1} =  \pm (1 + \frac{1}{r - 1}).$$ 
Thus by integration of the latter equation we have the solutions (\ref{35}). 

\vspace{0.1cm}
For item (3), if $b \neq 0$ the integral in (\ref{37}) is a solution of the timelike radial  geodesic equation 
if, and only if, the equation (23) holds 
for $L = 0$. That can be checked using equation (\ref{37}) and taking the derivative of (\ref{36}).
\end{proof}

\begin{corol}\label{71}
Assuming that the relativistic mechanic for the gravitation is asymptotically the Newtonian mechanic, one has that  
$\displaystyle lim_{r \rightarrow +\infty} \frac{dr}{d\eta} = 0$.  Then  $b^{2} = 1$ 
and the velocity in the $t$-time for any timelike radial geodesic is given by
\begin{equation}\label{38}
v(t) = \frac{dr}{dt} = \frac{dr}{d\eta} \; \frac{d\eta}{dt} = \pm h(r(t)) \sqrt{1 - h(r(t))}.
\end{equation}
\end{corol} 

\begin{proof} $b^2 = 1$ follows directly from equation (\ref{37}).  Equation (\ref{38}) follows from 
$b^2 = 1$, $\frac{dr}{d \eta}= \pm \sqrt{ 1 - \frac{r-1}{r}}$  and $h(r(t)) = \frac{r-1}{r}$.
\end{proof}

Next we conclude that the $v(t)$-velocity function assumes a maximum value when $r = 3$.

\vspace{0.1cm}
\begin{corol}\label{08}
The function $\phi$ defined through of equation (\ref{38}), namely, 
$$2 \phi(t) = v^{2}(t) = h^{2}(r(t))  [1 - h(r(t))]$$ assumes a maximum value when $h(r) = 2/3$, it 
means when $r = 3$,  and the maximum value is $\phi_{I} = 2/27$. 
\end{corol}

Next lemma follows from a simple integration.
\begin{lemma}\label{40}
In Newtonian mechanic the time $\hat{t}(r)$ associated to a radial motion with relative speed \ 
$\tilde{v} = \displaystyle \sqrt{\frac{2MG}{rc^{2}}}$ is given by the relation 
\begin{equation}\label{39}
\frac{dr}{d\hat{t}} = K \frac{1}{\sqrt{r}} \; \; \; \mbox{ where } \; \; \; K = \frac{2MG}{c^{2}}.
\end{equation}
Then, for $K = 1$ the time takes the form of 
$$\hat{t}(r) = \frac{2}{3}(r^{3/2} - r_{0}^{3/2})\; \; \mbox{ for initial condition } \; \; \hat{t}(r_{0}) = 0.$$
\end{lemma}

The following theorem establishes a relationship between the Schwarzschild time and the Newton mechanic time.
\begin{theor}\label{699}
In a radial geodesic $\alpha(\eta)$ the time of Schwarzschild $t$ and the time of Newton mechanic $\hat{t}$ are related 
by the equation
\begin{equation}\label{39}
\frac{d\hat{t}}{dt} = h(r) \sqrt{\frac{1 - b^{2} h(r)}{1 - h(r)}}.
\end{equation}

Moreover, if $b^{2} = 1$ (that means the timelike radial geodesic line is asymptotically Newtonian), the 
Newton-time $\hat{t}$ is a proper parameter for those geodesic. 
\end{theor}
\begin{proof}
Using the equation (\ref{37}), equation of Lemma \ref{40}, namely $\frac{dr}{d\hat{t}} = \frac{1}{\sqrt{r}}$, 
formula (\ref{39}) follows from the relation $r = \frac{1}{1 - h}$. 

If $b^{2} = 1$  it follows that $d\hat{t} = d\eta$. In fact, first from equation (\ref{39}) one has 
$\frac{d\hat{t}}{dt} = h(r)$ and, since $\frac{d \eta}{dt} = h(r)$, one arrives $d\hat{t}/dt = d\eta/dt$ and so 
$d\hat{t} = d\eta$. Thus, from Definition \ref{41}, it follows that $\hat{t}$ is a proper parameter for $\alpha$.
\end{proof}

In the next subsections of this section we retake the geodesic equation in Theorem \ref{30} and we prove formulas 
(21), (22) and (23).

\vspace{0.1cm}
\subsection{Proof of Equation (23)}
\begin{proof} 
It notes that a part of computations for proving Equation (23) was already made in the proof of Proposition \ref{42}. 
From equation of geodesic considering the $r$-variable one has that the geodesic equation reduces to 
$$r'' + \Gamma_{00}^{1} (t')^{2} + \Gamma_{11}^{1} (r')^{2} + \Gamma_{22}^{1} (\varphi')^{2} + \Gamma_{33}^{1} (\theta')^{2} = 0$$ 
since that the other Christoffel symbols vanish identically. From $g_{22} = r^{2}$ and $g_{33} = r^{2} \sin^{2} \varphi$ 
we have that   
$$\Gamma_{22}^{1} = -(r - 1) \; \; \; \mbox{ and } \; \; \; 
\Gamma_{33}^{1} = -(r - 1) \sin^{2} \varphi.$$
Then, since $\Gamma_{00}^{1} = \displaystyle \frac{r-1}{2r^3}$, $\Gamma_{11}^{1} = - \displaystyle \frac{1}{2r(r - 1)}$ 
and $\eta$ is a effective parameter it follows finally the equation (23). 
\end{proof}

\vspace{0.1cm}

\subsection{Proof of the Equation (21)}
\begin{proof} 
Equation (21) comes from considering the $\theta$-variable. That means the equation 
$\theta'' + \sum_{i,j = 0}^{3} \Gamma_{ij}^{3} \displaystyle \frac{dx^{i}}{d \eta} \frac{dx^{j}}{d \eta}= 0$. Next we 
need to compute the non-null symbols $\Gamma_{ij}^{3}$. Since $g_{22} = r^{2}$ and $g_{33} = r^{2} \sin^{2} \varphi$   and
$$\Gamma_{ij}^{3} = \frac{1}{2}  g^{33}\left\{\frac{\partial g_{i 3}}{\partial x^{j}} + 
\frac{\partial g_{j 3}}{\partial x^{i}} - \frac{\partial g_{ij}}{\partial x^{3}}\right\},$$ 
it follows that the only non-null Christoffel symbols are  $\Gamma_{13}^{3} = \frac{1}{r}$ and 
$\Gamma_{23}^{3} = \frac{\cos \varphi}{\sin \varphi}$.
For which the equation becomes to
$$(r^{2} \sin^{2} \varphi) \theta'' + (2r \sin^{2} \varphi) r' \theta' + (2 r^2 \sin \varphi \cos \varphi) \varphi' \theta' = 0.$$
That means $(r^{2} \sin^{2} \varphi \; \theta')' = 0$
then, Equation (21) holds for some constant $L$. 
\end{proof}

\vspace{0.1cm}
\subsection{Proof of the Equation (22)}  For proving Equation (22) we will consider next a new orthonormal frame 
$\F(M)$, taken on the fact that Equation (22) involves variations in the $\theta$ and $\varphi$-coordinates, those associated to spherical 
coordinates. 
\begin{dfn}
Let $\F(M) = \{\hat{\partial}_{t}, \hat{\partial}_{r}, \hat{\partial}_{\varphi}, \hat{\partial}_{\theta}\}$ be an 
adapted orthonormal frame to $M$, defined by the fields: 
$$\hat{\partial}_{t} = \frac{1}{\sqrt h} \partial_{t}, \; \; \; \hat{\partial}_{r} = \sqrt{h}\partial_{r}, \; \; \; 
\hat{\partial}_{\varphi} = \frac{1}{r} \partial_{\varphi}, \; \; \; 
\hat{\partial}_{\theta} =  \frac{1}{r \sin \varphi} \partial_{\theta}.$$  
\end{dfn}

\vspace{0.3cm}
Therefore, since $ds^{2}(M) = -h dt^{2} + \frac{1}{h} dr^{2} + r^{2} (d\varphi^{2} + \sin^{2} \varphi d\theta^{2})$, we can take the 
spherical coordinates for $\real^{3}$ as $\hat{\Gamma} : ]1, +\infty[ \times \real^{2} \to \real^{3}$ 
defined by  
$$\hat{\Gamma}(r,\varphi,\theta) = r N(\varphi,\theta) \; \; \mbox{ where } \; \; 
N(\varphi,\theta) = (\sin \varphi \cos \theta, \sin \varphi \sin \theta, \cos \varphi),$$  to 
 identify 
$$\hat{\partial}_{\varphi} \equiv N_{\varphi} \; \; \; \mbox{ and } \; \; \;
\hat{\partial}_{\theta} \equiv \frac{1}{\sin \varphi} N_{\theta}.$$  

\vspace{0.2cm}
Now, since the geodesic condition is equivalent to $\alpha''$ to be orthogonal to tangent vector, 
next we will take a model for $M$, denoted by $F(M)$, where the geodesic condition will be understood as 
$\alpha''(p) \in (T_{p}F(M))^{\perp}$. That means that we are taking a foliation of $\euclidean$ by spheres 
centered  in $(0,0,0,0)$ and radius $r(\eta)$.

\vspace{0.2cm}
In the proof of the following lemma we will prove Equation (22). In particular, in the process of the proof, Equation (21) 
will be again proved.

\begin{lemma}\label{45}
Let $\hat{\gamma}(\eta) = \hat{\Gamma}(r(\eta),\varphi(\eta),\theta(\eta))$ be a curve in $\real^{3}$ associated to a geodesic line in 
$M$. Writing $\hat{\gamma}(\eta) = r(\eta) N(\varphi(\eta),\theta(\eta))$ then the geodesic condition for $\hat \gamma : I \to F(M)$, 
implies that 
$$\hat{\gamma}''(\eta) = Ac(\eta) N(\varphi(\eta),\theta(\eta)).$$   
\end{lemma} 

\begin{proof}

In order to do that we start
observing that  since $$\hat{\gamma}' = r' N + r N_{\varphi} \varphi' + r N_{\theta} \theta'$$ it follows that   
$\lpr{\hat{\gamma}'}{\hat{\gamma}'} = (r')^{2} + r^{2} \left[(\varphi')^{2} + (\theta')^{2} \sin^{2} \varphi  \right]$. 
Therefore, 

\begin{enumerate}
\item From $\lpr{\hat{\gamma}'}{N} = r'$ it follows that 
$r'' = \lpr{\hat{\gamma}''}{N} + \lpr{\hat{\gamma}'}{N_{\varphi}} \varphi' + \lpr{\hat{\gamma}'}{N_{\theta}} \theta'$ then, 
$$\lpr{\hat{\gamma}''}{N} = r'' - r[(\varphi')^{2}+ (\theta')^{2} \sin^{2} \varphi].$$ 

\item From $\lpr{\hat{\gamma}'}{N_{\varphi}} = r \varphi'$ it follows that 
$$\lpr{\hat{\gamma}''}{N_{\varphi}} + \lpr{\hat{\gamma}'}{N_{\varphi \varphi} \varphi'} + 
\lpr{\hat{\gamma}'}{N_{\varphi \theta} \theta'} = r' \varphi' + r \varphi''.$$
Now, because $N_{\varphi \varphi} = - N$ implies that $\lpr{\hat{\gamma}'}{N_{\varphi \varphi} \varphi'} = -r'\varphi'$, and since 
$N_{\varphi \theta} = \cos \varphi (-\sin \theta, \cos \theta, 0)$, we obtain that $\lpr{\hat{\gamma}''}{N_{\varphi}} = 0$ if 
and only if
$$r \varphi'' + 2 r' \varphi' = r \; \theta'^{2} \sin \varphi \; \cos \varphi.$$ 
The latter equality implies  $(r^{2} \varphi')' =  r^2  \; \theta'^{2} \sin \varphi  \cos \varphi$, 
that is the equation (22).

\vspace{0.2cm}
\item In addition, since $\lpr{\hat{\gamma}'}{N_{\theta}} = r \theta' \sin^2 \varphi$ and 
$N_{\theta \theta} = - \sin \varphi ( \cos \theta,  \sin \theta, 0)$  we have that $\lpr{\hat{\gamma}''}{N_{\theta}} = 0$ 
if, and only if, $(r^2 \theta' \sin^2 \varphi)' = 0$. That means equation (21) holds if, and only if, 
$\lpr{\hat{\gamma}''}{N_{\theta}} = 0$. But, it is well known that each geodesic curve of a submanifold has acceleration vector field 
orthogonal to its correspondent tangent bundle. Therefore, the equation (21) holds. 
\end{enumerate} 
\end{proof}

After Lemma \ref{45} we can say that we have finished the proof of Theorem \ref{30}. 

Furthermore, Lemma \ref{45} also allows us to 
conclude the following result (that is closed related with our liftings to define the generalized Lorentz group) which shows 
that the effective parameter $\eta$ is necessary to obtain that the lifting of a curve of $\euclidean$ is a geodesic line in $M$. 
\begin{theor}
Let $\hat{\gamma}(\eta) = \hat{\Gamma}((r(\eta),\varphi(\eta),\theta(\eta))$ be a curve in $\euclidean$ associated to 
a geodesic line in $M$. Then 
\begin{equation}\label{60}
\hat{\gamma}''(\eta) = Ac(\eta) N(\varphi(\eta),\theta(\eta)) = 
\frac{1}{r} [r'' - r((\varphi')^{2} + (\theta')^{2} \sin^{2} \varphi)] \hat{\gamma}(\eta).
\end{equation}
\end{theor}

\begin{proof} 
The proof follows from Lemma \ref{45}, from 
$N= \displaystyle \frac{1}{r} \hat{\Gamma}(r(\eta),\varphi(\eta), \theta(\eta))= \frac{1}{r} \hat{\gamma}(\eta)$ 
and from of the fact that 
$\left <\gamma'', N \right> = r'' - r[(\varphi')^{2} + (\theta')^{2} \; \sin^{2} \varphi]$. 
\end{proof}

\vspace{0.1cm}
Continuing with our study of the geodesic lines on $M$ and still using the model $F(M)$, 
we next consider the case when the geodesic is non-radial with effective parameter.  That is made in the following subsection.

\subsection{A Cauchy problem associated to non-radial geodesic lines in $M$} We start remembering that if a point 
$p \in M$ and a non-radial velocity vector $V(p) \in T_{p}M$ are given, we can find a 
geodesic line $\alpha(\eta) = (\beta(\eta), \gamma(\eta)) \in P \times \C$ for the conditions $\alpha(0) = p$ and 
$\alpha'(0) = V(\alpha (0))$, by solving the system of geodesic equations. 
Therefore, taking the projection into $\euclidean$ and writing  $p = (t_{0},r_{0},\varphi_{0},\theta_{0})$ and 
$V = V^{0} \hat{\partial}_{t}(p) + V^{1} \hat{\partial}_{r}(p) + 
V^{2} \hat{\partial}_{\varphi}(p) + V^{3} \hat{\partial}_{\theta}(p)$, one obtains a Cauchy problem restricted to 
$\gamma(\eta)$, namely, the spatial part of geodesic equation, since by equations (21) and (22) we will have that $\hat{\gamma}''(p)$ 
needs to be parallel to position vector $\hat{\gamma}(p)$. Then we have the equations:

\begin{equation}\label{51}
\hat{\gamma}'(\eta) = r' N(\eta) + r \varphi' N_{\varphi} + r \theta' N_{\theta} \; \;  
\mbox{ with } \; \; \hat{p} = \hat{\gamma}(0)  \; \; \mbox{ and } \; \;  \hat{\gamma}'(0) = \vec{v}.
\end{equation} 

\vspace{0.1cm}
Since $\hat{\gamma}'(0) = v_{0}$ is non-radial, the vector subspace $W = \mbox{Span}[\hat{\gamma}(0),\hat{\gamma}'(0)] 
\subset \euclidean$ contains the trajectory traced by this part of the geodesic. 
The radial projection of $\gamma(\eta)$ into the unit sphere $S^{2} = \{w \in \euclidean: \lpr{w}{w} = 1\}$ traces an arc of a 
great circle in this sphere. Now, by a linear isometry, more explicitly a spherical rotation, we can assume a coordinates system 
where $\varphi = \pi/2$ and so in these new coordinates the geodesic equation reduces to equations (20), (21) and (23), 
with an unique constant $L$, that is given by the initial conditions.

\begin{prop}\label{12}
In rectangular coordinates $\hat{\gamma} = (x,y,z)$ for $\euclidean$ the equation (\ref{60}) can be written as follows 
\begin{equation}\label{62}
\hat{\gamma}''(\eta) - ac(\eta) \hat{\gamma}(\eta) = 0, \; \mbox{ where} \; \; 
\hat{\gamma}(0) = \hat{\gamma}_{0} \; \mbox{ and } \; \hat{\gamma}'(0) = {v}_{0} \; \mbox{ are given.}
\end{equation}

By the Theorem of Existence and Unicity of solutions for the ordinary differential equation one has for each initial 
condition an unique solution of the Cauchy problem given above. 
\end{prop}

Next we observe that since a general solution of (\ref{62}) depends of two parameters $k_{1}, k_{2} \in \real$ and 
two functions which are solutions of the equation $X''(\eta) - ac(\eta) X(\eta) = 0,$ it follows that if 
$X = (x,y,z)$ then the set of solutions $\{x(\eta),y(\eta),z(\eta)\}$ form a linearly dependent set over $\real$. 
Then we have the following consequence:

\begin{corol} 
Each non-radial parametric geodesic line $\alpha(\eta)$ in $M$ equipped with effective parameter $\eta$, is an extension of a plane 
curve $\hat{\gamma}(\eta)$ in $\euclidean$ adapted to the initial conditions. That means, an extension of a curve $\hat{\gamma}(\eta)$, which  in 
rectangular coordinates $\hat{\gamma} = (x,y,z)$ is such that
\begin{equation}
 \left\{\begin{matrix}
x''(\eta) + \lambda(\eta) x = 0 \; \; \mbox{ where } \; \; x(0) = x_{0} \; \; \mbox{ and } \; \; x'(0) = {v}_{1} \\
y''(\eta) + \lambda(\eta) y = 0 \; \; \mbox{ where } \; \; y(0) = y_{0} \; \; \mbox{ and } \; \; y'(0) = {v}_{2} \\ 
\ z''(\eta) + \lambda(\eta) z = 0 \; \; \mbox{ where } \; \; z(0) = 0 \; \; \mbox{ and } \; \; z'(0) = 0, 
\end{matrix}\right.
\end{equation} 
where $\lambda(\eta) = - ac(\eta)$. 

With these coordinates $\cos \varphi = 0$, and so we say that the geodesic line 
$\alpha(\eta)$ is an {\it equatorial geodesic}.   
\end{corol}

In the following we establish one of our mean results in this section. In fact, we find solutions for 
equations (20), (21) and (23) for an equatorial geodesic, that is  $\varphi = \pi/2$, with the condition 
$r'(\eta) \equiv 0$ (equivalently $r(\eta) > 1$ is a constant).

\begin{theor}\label{68}
An equatorial geodesic line $\alpha(\eta)$ parametrized with effective parameter $\eta$ such that $r'(\eta) \equiv 0$, 
is established by the relations:
\begin{equation}
r^{3} - 2 L^{2}(r - 1)^{2} = 0 \; \; \mbox{ and the Kepler-Newton orbit speed } \; \; v_{orb} = \pm \sqrt{\frac{1}{2r}}.
\end{equation} 

Moreover, using the Newton-time $\hat{t}$ it follows that $v_{orb}(\hat{t}) = v_{orb}(t) = \pm 1/\sqrt{2r}$. That means $\hat{t} = t$. 
Since $r$ is constant, then the time $t = \hat{t}$ is a proper parameter of those geodesic lines. 
\end{theor}
\begin{proof}
As by hypothesis $r'(\eta) \equiv 0$ one has that the radial position is constant, it means 
$r(\eta) = r \in \real$ for $\eta \in I$ with $r$ being a constant.  In addition, since $\varphi = \pi/2$, 
the geodesic equations (20), (21), (23) reduce, respectively, to  
\begin{equation}\label{67}
\frac{dt}{d\eta} = \frac{r}{r - 1}, \; \; \; \; \; \; \; \frac{d \theta}{d\eta} = \frac{L}{r^{2}} \; \; \; \; \; \; \; { and} \; \; \; \; \; \; \;
\frac{1}{2r(r - 1)} = \frac{L^{2}}{r^{4}} \; (r - 1),
\end{equation} 
where we have replaced the equation (21) in equation (23) to get the latter equation in (\ref{67}).
Then, the angular momentum $L$ and the constant radial position $r$ are related by the formula:
$$L = \frac{\pm r}{r - 1} \sqrt{\frac{r}{2}} \; \; \; \; \mbox{or equivalently } \; \; \; \; r^{3} - 2 L^{2}(r - 1)^{2} = 0.$$ 

Next, we define the function that we will be called {\it the  Kepler-Schwarzschild orbital speed of a planet}, by 
$$v_{orb}(t) = r \frac{d\theta}{d\eta} \frac{d \eta}{dt}=  \pm h(r) \frac{L}{r},$$ where $t$ is the Schwarzschild time. 
 Then, from $h = 1 - 1/r$ and $L$ defined above, it follows that
$$v_{orb}(t) = \pm \sqrt{\frac{1}{2r}}.$$ 

\vspace{0.1cm}
Now, we observe that the orbit obtained is an uniform circular motion, therefore, 
from $v_{orb}(t) = v_{orb}(\hat{t})$ it follows that $v_{orb}(\hat t) = \pm 1/\sqrt{2r}$. This latter implies 
that the times are the same and the Newton's time $\hat t$ is a proper parameter of those geodesic lines.
\end{proof}

\subsection{Kepler-Bohr-Einstein} We start this subsection saying that the main results of Section 3 are 
given by the Theorem \ref{699} and \ref{68}.  In fact, it based on the Corollary 3.17, each geodesic line is either a 
radial geodesic or an equatorial geodesic. In both types of geodesics, by Theorem  \ref{699} and \ref{68}, 
its Newton's times are their proper parameters. In addition, the speed functions of Schwarzschild and of Newtonian agree each other when the geodesic line is an uniform 
circular  motion 
(or equatorial),  the function $v(t) = \pm \sqrt{1/(2r)}$ and  when $r \rightarrow 1$ we have $v(t) = \sqrt{2}/2$.

\begin{dfn}
 The function $\Phi(r)$ given by 
\begin{equation}\label{69}
\frac{1}{2} \left(\frac{dr}{dt}\right)^{2} = \Phi(r) = \frac{1}{2} h^{2}(r) (1 - b^{2} h(r)), 
\end{equation} 
we will be called of Schwarzschild-Bohr impedance for the electronic system of nucleus of the hydrogen atom. 

We will take as the mass of our ideal hydrogen atom, namely $\check{M}$, that established by  the equation

\begin{equation}\label{70}
\check{M} \phi_{I} = \frac{1}{\hat{\beta}} \; m_{e},
\end{equation} 
where $m_{e}$ denotes the mass of one electron. That means, $\check{M} = \frac{27}{2} (137 ) m_{e}$.
\end{dfn}

We note in particular that this definition is, indeed, the Fermi's definition of the energy of a curve, 
applied in our context. Moreover, one observes that equation (\ref{38}) in Corollary \ref{71} is a particular case of 
equation (\ref{69}), when $b = 1$.

\vspace{0.1cm}

\begin{prop}
It follows from Equation (\ref{70}) that the mass equivalent to the energy of the hydrogen atom is 
$$\check{M} = \left(\frac{r_{I}}{r_{0} \hat{\beta}}\right)^{3} \Delta m.$$ 
Very close to actual experimental value is given taking the mass of nucleus of hydrogen atom   
$$136 = \frac{1 - \hat{\beta}}{\hat{\beta}} \; \; \mbox{which implies } \; \;  
M_{\mbox{proton}} = \frac{136}{137} \check{M} = 1836 \times m_{e}.$$ 
\end{prop}

\vspace{0.2cm}
The experimental value for the mass of one neutron is $M_{\mbox{neutron}} = 1.67492749804 \times 10^{-27} kg$. 
Our ideal mass of the hydrogen atom is $\hat{M} \approx 1,684 \times 10^{-27} kg$, and so the ratio between these two 
values is 
$$ \frac{M_{\mbox{neutron}}}{\hat{M}} \approx \frac{1675}{1684} \approx 0,994655.$$ 

\vspace{0.2cm}
Finally, we note that using the formula $b^{2} = -\lpr{\alpha'}{\alpha'}$ and the non-ideal impedance function 
given in formula (\ref{69}), namely, $\frac{dr}{dt} = h(r)\sqrt{1 - b^{2} h}$, it follows that 
the maximal value for the energy (or impedance) of the curve is $\phi_{max} = \displaystyle \frac{2}{27 b^{4}}.$
So we can assume a value closer to the experimental value of hydrogen's mass, taking the constant 
$$b^{2} = 1 - \delta = \sqrt{\frac{1849,5}{1836,1}} \approx 0,99637 \; \mbox{ and } \; 
\sqrt{\frac{136}{137}} - \sqrt{\frac{1849,5}{1836,1}} = -0,00002713317.$$ 

\vspace{0.1cm}

\section{A model $(M',\mathbf{g})$ non Ricci flat and its Einstein Gravitational Tensor} 

We start this section taking the model $(M, \mathbf{g})$ from the previous section and we will construct 
a new model  $(M',\mathbf{g})$ not flat, with Ricci curvature  not null and which will satisfy the Einstein field equation. 
It has to be noticed now that our new model is different from the model described by O'Neill in the reference \cite{O'N}, 
that called the Black hole.
 
\vspace{0.1cm}
It begins writing the energy equation for a curve $\alpha(\eta)$ in $(M,\mathbf{g})$ as follows 
\begin{equation}
(a(\alpha))(\eta) = -b^{2}(\eta) + c^{2}(\eta),
\end{equation} 
where $b^2$ and $c^2$ are given by $-b^{2}(\eta) = -h (t')^{2} + \frac{1}{h} (r')^{2}$ \ and \ 
$c^{2}(\eta) = r^{2}[(\varphi')^{2} + (\theta')^{2} \sin^{2} \varphi]$.  So, the energy function takes the form
$$a(\alpha) = \lpr{\alpha'}{\alpha'} = -h (t')^{2} + \frac{1}{h} (r')^{2} + r^{2}[(\varphi')^{2} + (\theta')^{2}  \sin^{2} \varphi].$$ 

Now we observe that the curve $\alpha(\eta)$ can be seen as sum of two curves, one of them $\beta(\eta)$ in the 
Schwarzschild half plane $P_{I} : (t,r) \in \real \times ]1,+\infty[$ and the other   
$\overline{\gamma}(\eta)$ in the lightcone $\C \subset \mink$, where the lightcone is parametrically given by 
$$\Gamma(\rho,\varphi,\theta) = \rho L(\varphi,\theta) \; \; \mbox{ for } \; \; 
L(\varphi,\theta) = (1, \cos \theta \sin \varphi, \sin \theta \sin \varphi, \cos \varphi)$$ 
with linking given by $\rho = r$. Then we have first the following lemma.

\begin{lemma}
Let $\spac$ be the semi-Euclidean space with signature $(-1,-1,1,1)$. For $h(r) = \displaystyle \frac{r - 1}{r}$ 
and $\phi(r) = r + \ln(r -1)$ the Lorentzian parametric surface 
$$f_{I}(t,r) = 2\sqrt{h(r)} \left(\cos (\frac{t}{2}), \sin  (\frac{t}{2}), \cos  (\frac{\phi}{2}), \sin  (\frac{\phi}{2}) \right)$$ 
has induced metric tensor given by 
$$ds^{2}(P_{I}) = \frac{1 - r}{r} dt^{2} + \frac{r}{r-1} dr^{2}.$$ 

This model is constructed using the solution of a null radial geodesic line, obtained in Corollary \ref{72} item (2), namely,  
$t(r) = r_{0} \pm \phi(r)$ for these null lines.
\end{lemma}

Since $\lim_{r \to 1} h(r) = 0$  and the functions $\cos$ and $\sin$ are limited, we have the following corollary.
\vspace{0.1cm}
\begin{corol}
The Lorentzian parametric surface $f_{I}(t,r)$ given by 
$$f_{I} : \real \times ]1,+\infty[   \mapsto 2\sqrt{h(r)} \left(\cos  (\frac{t}{2}), \sin  (\frac{t}{2}), 
\cos  (\frac{\phi}{2}), \sin  (\frac{\phi}{2})\right) \in \spac$$ 
can be continuously extended to $r = 1$ where 
$$\lim_{r \to 1} f_{I}(t,r) = 0 \; \; \mbox{ for all } \; \; t \in \real.$$ 
\end{corol}

\vspace{0.2cm}
So from the above it follows that in the product manifold  $P_{I} \times \C \subset \spac \times \mink$, 
using the identification $r = \rho$, we can see the Schwarzschild spacetime as a Lorentzian immersed parametric 
$4$-surface in $\spac \times \mink$. For the next notation, we will prefer to use the directed sum notation.

\vspace{0.3cm}

Now we focus in  the case when $r \in ]0, 1[$ and we will construct a model of Lorentzian spacetime 
which is not a Ricci flat space.
\vspace{0.1cm}
\begin{dfn}
For the strip $\real \times ]0,1[$, let $\mathbf{g}$  be the Lorentzian metric tensor of the surface 
\begin{equation}
f_{I'}(t,r) = 2\sqrt{\frac{1 - r}{r}\; } \left(\cos(\frac{t}{2}), \sin (\frac{t}{2}), 
\cos (\frac{\phi}{2}), \sin (\frac{\phi}{2}) \right)
\end{equation} 
where, now, we are taking $\phi(r) = 1 - r - \ln (1 - r)$ from Corollary \ref{72}, for $b^{2} = 0.$  
\end{dfn}

\begin{theor}\label{010}
The map 
\begin{equation}\label{011}
F_{I} : (t,r,\varphi,\theta) \in \real \times ]1,+\infty[ \times \real^{2} \; \mapsto \; f_{I}(t,r) \oplus 
r L(\varphi,\theta) \in \spac \oplus \mink
\end{equation} 
is an isometric proper immersion of the Schwarzschild spacetime. It is a Ricci flat space.

The map 
\begin{equation}\label{012}
F_{I'} : (t,r,\varphi,\theta) \in \real  \times ]0, 1[ \times \real^{2} \; \mapsto \; f_{I'}(t,r) \oplus 
r L(\varphi,\theta) \in \spac \oplus \mink
\end{equation} 
gives a parametric Lorentzian spacetime, which we will denote by $(M',\mathbf{g})$,  with line element given by 
\begin{equation}
ds^{2} = \frac{r - 1}{r} dt^{2} + \frac{r}{1 - r} dr^{2} + r^{2} [(d\varphi)^{2} + \sin^{2} \varphi \; (d\theta)^{2}].
\end{equation}  
\end{theor} 

\vspace{0.2cm}

Here we note that the first statement in Theorem \ref{010} about the flat Ricci condition for the 
Schwarzschild spacetime given in (\ref{011}) it can be found in reference \cite{O'N}.  Therefore, in our next steps we 
will focus our attention, in proving that the spacetime $(M',\mathbf{g})$ has  non vanishes Ricci tensor. 

\vspace{0.2cm}

We begin calculating the Ricci curvature tensor $Ric$ of $M'$ and the Einstein's gravitational tensor $\mathbf{G}$ defined by 
\begin{equation}\label{83}
\mathbf{G} = Ric - \frac{1}{2} S \mathbf{g},
\end{equation}
where $S$ denotes the scalar curvature. 

\vspace{0.2cm}
We will think $\mathbf{G}$ as the rules or precepts, for the gravitation and for the Bohr hydrogen atom. To obtain a 
realistic meaning for the tensor $\mathbf{G}$ we will use the Einstein equation $T = k \mathbf{G}$, where 
$T$ is the energy-moment tensor and $k$ is a convenient constant. 

\vspace{0.2cm}

The following lemma is the Exercise (5) in O'Neill book \cite{O'N}, page 156. It has a simple proof from the definition of $\mathbf{G}$. 
 
\begin{lemma}\label{73}
For each abstract surface $(U,ds^{2})$ such that $ds^{2} = E(r) dt^{2} + G(r) dr^{2}$, it follows the following statements:

\vspace{0.1cm}
\begin{enumerate}
\item The gradient function for $r$ is given by  \ $grad \; r = \frac{1}{G} \partial_{r}$. 

\vspace{0.1cm}
\item The Hessian tensor is given by  
\begin{equation}\label{77}
H^{r}(\partial_{t},\partial_{t}) = \frac{E_{r}}{2G} \; \; \;  \; \mbox{ and } \; \; \; \; H^{r}(\partial_{r},\partial_{r}) = 
\frac{-G_{r}}{2G}.
\end{equation} 
Observing that the Hessian is a diagonal symmetric tensor, $H^{r}(\partial_{t},\partial_{r}) = 0$. 

\vspace{0.1cm}
\item The Laplacian operator takes the form 
\begin{equation}\label{78}
\Delta  r = \frac{1}{2G}\left(\frac{E_{r}}{E} - \frac{G_{r}}{G}\right).
\end{equation}
\end{enumerate}
\end{lemma} 

\vspace{0.2cm}
Now we take again from \cite{O'N} page 211, the Corollary 43 applied for $M'$,  which we can see as the warped 
product of the Lorentzian strip $P_{I'} = (\real \times ]0,1[, \mathbf{g})$ and the Euclidean sphere 
$S^{2} \subset \real^{3}$, that means $M' = P_{I'} \times_r S^2$.  The notation $\mathfrak{L}(X)$ in the next results means 
the tangent vector field to the manifold $X$.

\begin{lemma}\label{79}
For the Ricci operator on $M' = P_{I'} \times_r S^2$, using the dimension 2 of the surfaces $P_{I'}$ and $S^{2}$, it follows: 

\begin{enumerate}
\vspace{0.1cm}

\item $Ric(X,Y) = Ric^{ P_{I'}}(X,Y) - \displaystyle \frac{2}{r} H^{r}(X,Y)$ for all horizontal fields 
$X,Y \in \mathfrak{L}( P_{I'})$. 
\vspace{0.1cm}
\item $Ric(X,W) = 0$ for all $X\in \mathfrak{L}(P_{I'})$ and all $W \in \mathfrak{L}(S^{2})$. 
\vspace{0.1cm}
\item $Ric(V,W) = Ric^{S^{2}}(V,W) - \lpr{V}{W} r^{*}$ for all vertical fields $V,W \in \mathfrak{L}(S^{2})$, where 
$$
r^{*} = \frac{\Delta r}{r} + \frac{\left < grad \ r, grad \ r\right>}{r^2}$$
and $\Delta r$ is the Laplacian on $P_{I'}$.
\end{enumerate}
\end{lemma} 

The proof the following lemma is analogous to the construction of the Schwarzschild spacetime and the Lemma 1 in page 366 of 
the reference \cite{O'N}.  

\begin{lemma} For the warped product $M' = P_{I'} \times_{r} S^{2}$ the Gauss curvature is given by 
$K(P_{I'}) = - \displaystyle \frac{1}{r^{3}}$. Therefore:

\begin{enumerate}
\item Since $dim(P_{I'}) = 2$ it follows that $Ric^{P_{I'}}(X,Y) = K_{P_{I'}} (X, Y)$. Moreover, from Lemma \ref{73} item (2), one has
$$Ric(X,Y) = 0 \; \; \mbox{ for each } \; \; X,Y \in \mathfrak{L}(P_{I'}).$$

\item Furthermore,  $r^*$ is given by
\begin{equation}\label{81}
 r^{*} = \displaystyle {\frac{1}{2 r G}} \left\{\frac{E_{r}}{E} - \frac{G_{r}}{G}\right\} + \frac{1}{r^{2} G} = \frac{-1}{r^{2}},
\end{equation}
and so,  $Ric(V,W) = \displaystyle \frac{2}{r^{2}} \lpr{V}{W},$ for $V,W \in \mathfrak{L}(S^{2})$.
\end{enumerate} 
\end{lemma}

\begin{proof}
First we will compute $K(P_{I'})$. Here we pay attention that in our model $M'$ the function 
$h = \displaystyle \frac{1-r}{r}$. Hence, the second fundamental form has $e = \sqrt{h(r)}$ and $g = \sqrt{1/h(r)}$. 
Now we apply the formula (2) in Proposition 44, page 81 of reference \cite{O'N}, and observing that signal of 
$G = \displaystyle \frac{1}{h(r)}$ is $\epsilon_{2} = 1$ and that $eg = 1$, it results in 
$$K(P_{I'}) = - \frac{\partial}{\partial r}\left(\frac{e_{r}}{g} \right) = 
- \frac{\partial}{\partial r} \left(\frac{-1}{2r^{2}} \right) = \frac{-1}{r^{3}}.$$

\vspace{0.1cm}
For proving item (1) we use formulas (\ref{77}) to obtain
$$H^{r}(\partial_{t},\partial_{t}) = \frac{1 - r}{2r} \left(1 - \frac{1}{r}\right)' = \frac{1 - r}{2r^{3}} \; \; \; \; 
\mbox{and} \; \; \; \; H^{r}(\partial_{r},\partial_{r}) = \frac{1 - r}{2r} \left(\frac{-r}{1 - r}\right)' = \frac{-1}{2r(1 - r)}.$$ 

Now, since that the volume form is preserved, $EG = -1$, the statement (1) of this lemma follows. Indeed, because 
the Gauss curvature is $K(P_{I'}) = -\displaystyle \frac{1}{r^{3}}$ and 
$Ric^{P_{I'}}(\partial_t, \partial_t) = \displaystyle \frac{1 - r}{r^4}$, 
it follows from item (1) of Lemma \ref{79} that 
$$Ric(\partial_{t},\partial_{t}) = \frac{-1}{r^{3}} \; \frac{-1 + r}{r} - \frac{2}{r} \; \frac{1 - r}{2r^{3}} = 0 \; \; 
{ and } \; \; Ric(\partial_{r},\partial_{r}) = \frac{-1}{r^{3}} \; \frac{r}{1 - r} - \frac{2}{r} \; \frac{-1}{2r(1 - r)} = 0.$$ 
Hence by bilinearity of the Ricci tensor it follows the statement (1). 

\vspace{0.3cm}
For item (2), we observe that since $\displaystyle \frac{E_{r}}{E} = - \frac{G_{r}}{G}$ using formula (\ref{78}) we obtain 
$$\Delta r = \frac{1 - r}{r} \; \frac{-1}{r(1 - r)} = \frac{-1}{r^{2}}$$
therefore it follows that
$$r^{*} = \frac{-1}{r^{3}} + \frac{1 - r}{r^{3}} = \frac{-1}{r^{2}}.$$ 

\vspace{0.2cm}
Finally, using $r^*$ above,  the Gauss curvature of spheres of $\real^{3}$  and  item (2) of Lemma \ref{79} for 
$V,W \in  \mathfrak{L}(S^{2})$, it follows
$$Ric(V,W) = \frac{1}{r^{2}} \lpr{V}{W} - r^{*} \lpr{V}{W} = \frac{2}{r^{2}} \lpr{V}{W}.$$ 
\end{proof}

From these lemmas above we have the main theorem of this section.

\vspace{0.2cm}
\begin{theor}
For the spacetime $(M',\mathbf{g})$, for each $X,Y \in \mathfrak{L}(P_{I'})$ and $V,W \in \mathfrak{L}(S^{2})$, 
one has $Ric(X,Y) = 0$ and $Ric(V,W) = \frac{2}{r^{2}} \lpr{V}{W}$. Therefore the Einstein's gravitational tensor 
$\mathbf{G}$ is given by  
\begin{equation}\label{84}
\mathbf{G}(X,Y) = \frac{-2}{r^{2}} \lpr{X}{Y}.
\end{equation}
\begin{equation}\label{85}
\mathbf{G}(V,W) = 0. 
\end{equation}
\end{theor}

\begin{proof}
We need to compute the scalar curvature $S$ of $M'$. For that, we use the frame 
$\{\hat{\partial}_{t}, \hat{\partial}_{r}, \hat{\partial}_{\varphi}, \hat{\partial}_{\theta}\}$ to obtain that
$S = \sum_{i = 0}^{3} \epsilon_{i} Ric(\hat{\partial}_{i},\hat{\partial}_{i}) = \frac{2}{r^{2}} + \frac{2}{r^{2}} = \frac{4}{r^{2}}$. 
Using now formula (\ref{83}) together with
$Ric(X,Y) = 0$ and $Ric(V,W) = \frac{2}{r^{2}} \lpr{V}{W}$ it follows formulas  (\ref{84}) and (\ref{85}).
\end{proof}

\vspace{0.2cm}
Since $\lim_{r \to 1} f_{I'}(t,r) = 0$ for all $t \in \real$, we can extend $M \cup M'$ to a  
connected space $M \cup S^{2} \cup M'$ with $S^2 \subset \mathbb R^3$,  where Kepler-Bohr orbits can be obtained for all 
$r \in ]1/2,+\infty[$, because according our Theorem \ref{68} the times $t$ and $\hat{t}$ are equal  and the velocity 
in these orbits satisfies $v_{orb} < 1$. 

\vspace{0.1cm}
\begin{corol} Taking a frame $\{U_{0},U_{1},U_{2},U_{3}\}$ from equation (\ref{84}) it follows that   
$$\mbox{ density} = \mathbf{G}(U_{0},U_{0}) = \frac{-2}{r^{2}}\lpr{U_{0}}{U_{0}} = \frac{2}{r^{2}} \; \; \; \mbox{ and } \; \; \;
\mbox{pressure} = \mathbf{G}(U_{1},U_{1}) = \frac{-2}{r^{2}}$$ 
are the non-null components of the Einstein gravitational tensor $\mathbf{G}$ whose 
contraction is $G_{0}^{0} + G_{1}^{1} = - \displaystyle \frac{4}{r^{2}}$. 
\end{corol}

\subsection{Geodesic Equations for $M'$ and Some Remarks}

\vspace{0.2cm}
In this subsection we  give our attention to the study of the geodesic equations for $(M', \mathbf{g})$ and we obtain results 
involving its solutions.  In fact, we start proving one of our main results in this paper, namely Lemma \ref{89}, which affirms 
that in our model $(M', \mathbf{g})$ and for $r < 1$ the existence of Kepler's orbitals are not possible. 
\vspace{0.1cm}

\begin{lemma}\label{89}
For $r(\eta) = r_{0} \in ]0,1[$ and $\varphi(\eta) = \pi/2$ for each $\eta \in I$  we cannot have circular geodesic motions.   
\end{lemma} 
\begin{proof}
We proceed by contradiction. It supposes that a geodesic line is possible and that the curve is given parametrically by
$$\alpha(\eta) = 2\sqrt{\frac{1 - r_{0}}{r_{0}}} \left(\cos \frac{t}{2}, \sin  \frac{t}{2}, \cos  \frac{\phi_{0}}{2}, 
\sin \frac{\phi_{0}}{2} \right) \oplus 
r_{0} \left(1, \cos \theta, \sin \theta, 0\right),$$
where $\phi_{0} = \phi(r_{0}) = 1 - r_0 - \ln(1 - r_0)$ is a constant and $t = t(\eta)$. Then taking 
$\alpha(\eta) = \beta(\eta) \oplus \gamma(\eta)$ where 

\vspace{0.2cm}
(i) $\beta(\eta) = 2\sqrt{-E_{0}}N(t(\eta))$ for $N(t) = (\cos \frac{t}{2}, \sin \frac{t}{2}, 
\cos \frac{\phi_{0}}{2}, \sin \frac{\phi_{0}}{2})$, and 

\vspace{0.2cm}
(ii) $\gamma(\eta) = r_{0} L(\theta(\eta))$ for $L(\theta) = (1,\cos \theta, \sin \theta, 0),$

\vspace{0.2cm}
it follows
$$\frac{1}{\sqrt{-E_{0}}} \beta''(\eta) = - \frac{1}{2} \left(\cos \frac{t}{2}, \sin \frac{t}{2}, 0, 0 \right) (t')^{2} + 
\left(-\sin \frac{t}{2}, \cos \frac{t}{2}, 0, 0\right) t''.$$

\vspace{0.2cm}
The geodesic condition $\lpr{\beta''}{f_{t}} = 0$  implies $t''(\eta) = 0$. 
Hence it follows $t(\eta) = a \eta + b$, with $a, b \in \mathbb R$. Therefore, since we have to preserve the directed future, 
we choice 
\begin{equation}\label{3}
\frac{dt}{d\eta} = \frac{r_{0}}{1 - r_{0}} \; \; \; \; \; \mbox{or} \; \; \; \; \; \frac{d\eta}{dt} = \frac{1}{r_{0}} - 1.
\end{equation}

The proof of geodesic equation (21) in our current hypotheses gives the same equal equation, namely,  
$(r_{0})^{2} \theta'(\eta) = L$ for some constant $L$. 

\vspace{0.2cm}

Now we will look for the linking equation for our current case.  The analogous of equation (23) here is 
$$r'' + \Gamma_{00}^{1} (t')^{2} + \Gamma_{11}^{1} (r')^{2} + \Gamma_{22}^{1} (\varphi')^{2} + \Gamma_{33}^{1} (\theta')^{2} = 0$$  
that from our assumptions becomes 
\begin{equation}\label{2}
\Gamma_{00}^{1} (t')^{2} + \Gamma_{33}^{1} (\theta')^{2} = 0. 
\end{equation} 
Since $E = - \vert h \vert$ and $G = 1/ \vert h \vert$ we obtain  $\Gamma_{00}^{1} = - \frac{1}{2}g^{11} (-|h|)_r$, that is
$$\Gamma_{00}^{1} = \frac{1 - r}{2r}\left(- \frac{\partial}{\partial r} \frac{-1 + r}{r}\right) = - \frac{1 - r}{2r^{3}} < 0.$$ 
Analogously, since  $g_{33} = r^{2}$,  one gets    
$$\Gamma_{33}^{1} = \frac{1 - r}{2r}(-2r) = -(1 - r) < 0.$$ 

Now since $\Gamma_{00}^{1}$ and $\Gamma_{33}^{1}$ are negative, it follows from equation (\ref{2}) that $t' = 0 = \theta'$. 
But from equation (\ref{3}) one has  $t' \neq 0$. Thus  we have a contradiction. Therefore, we can not have uniform circular motions 
for $0 < r < 1$.  
\end{proof} 

We will call the statement of Lemma \ref{89} {\it the absolute orbital impedance for $Z$-atoms}. 

\vspace{0.2cm}
Finally, we also note that the gravitational mass that \lq \lq cause" equal orbital singularity $r = r_{0}$ given by 
Kepler-Newton gravitational Theory and by Kepler-Bohr construction of the hydrogen atom Theory, is given by the equation 
\begin{equation}
\mathcal{M}G = \frac{e^{2}}{4 \pi \epsilon_{0} m_{e}} \; \; \mbox{ obtained from } \; \; 
r_{0} = \frac{e^{2}}{4 \pi \epsilon_{0} m_{e}c^{2}} = \frac{\mathcal{M}G}{c^{2}}. 
\end{equation} 
 
\vspace{0.3cm}

 Next we give an interpretation for our models $(M, \mathbf{g})$, $(M', \mathbf{g})$ and its junction as a geometric ideal hydrogen atom $H^{1}$. The external part of the atom will be described as the Schwarzchild spacetime $(M, \mathbf{g})$ which is Ricci-flat and $\mathbf{ G} =0$ (in a vacuum) and the inner part of the atom will be described by our Lorentzian spacetime $(M', \mathbf{g})$  which is non-Ricci flat with Einstein tensor $\mathbf{ G} \ne 0$ then with matter. 

For that we start observing that using second relation of Equation (\ref{19}) and the first of Equation (\ref{20}), it follows that the radius $r$ corresponding for the $n$-orbit is such that $r =  a_0 n^2$, where $a_0$ is the Bohr radius. Then Equation (\ref{28}) for $Z = 1$ and for $n = 1$ until $n > 1$, becomes to 
$$h \nu(1,n) = \frac{e^{2}}{8 \pi \epsilon_{0}} \left\{\frac{1}{a_{0}} - \frac{1}{a_{0} n^2}\right\},$$  
which we can see as the storage energy in a spherical capacitor. Hence we define: 
 
\begin{dfn}
Our geometric ideal hydrogen atom $H^{1}$ is the Bohr's hydrogen atom interpreted as a spherical electric capacitor 
of energy given by the Einstein's energy-stress tensor $\mathbf{T} = k \mathbf{G}$, where its nucleon is described by the model given by the Equation (\ref{012}), and where  
its exterior Kepler-Bohr orbital is modeled by equation (\ref{011}), obeying the postulates of Bohr written 
in the Subsection 2.5.
\end{dfn}

Our interest is now to explore more the solutions of  the geodesic equations under 
the restrictions $\varphi= \pi/2$ and $r(\eta) \in ]0, 1[$.

\begin{theor}
Let $\alpha(\eta) = \beta(\eta) \oplus \gamma(\eta)$ be a solution of the geodesic equations (\ref{29}), where 
it is assumed $\varphi = \pi/2$ and that the correspondent equations (20), (21) and (23), for the constants $L \neq 0$, are satisfied. 
Let $\E(\alpha) = \lpr{\alpha'}{\alpha'}$. Then
\begin{equation}\label{4}
\eta = \pm \int \sqrt{\frac{r^{3}}{r^{3} + (\E(\alpha)r^{2} - L^{2})(r - 1)} \;} \, dr.
\end{equation}

Equivalently, the equation for $r'(\eta)$ is given by 
\begin{equation}\label{5}
(r'(\eta))^{2} = \frac{P(r(\eta))}{r^{3}(\eta)} \; \; \; \mbox{ for } \; \; \;
P(r) = \left[1 + \E(\alpha)\right] r^{3} - \E(\alpha) r^{2} - L^{2}(r - 1).
\end{equation}
\end{theor} 

\begin{proof}
First we observe that $\lpr{\alpha'}{\alpha'} = \lpr{\beta'}{\beta'} + \lpr{\gamma'}{\gamma'}$, where 
$$\beta(\eta) = 2 \sqrt{\frac{1-r}{r}} \left(\cos \frac{t}{2}, \sin \frac{t}{2}, \cos \frac{\phi}{2}, \sin \frac{\phi}{2}\right) = 
2 \sqrt{\frac{1 - r}{r}}\; N$$
and
$$\gamma (\eta) = r \left(1, \cos \theta, \sin \theta, 0\right).$$ 
After a simple calculation one has $\left<N, N\right> = 0 = \left<N, N' \right>$ and $\left<N' , N'\right> = (\frac{\phi'}{2})^2 - 
(\frac{t'}{2})^2.$ Now, since $\phi (r(\eta)) = 1 - r(\eta) - \ln(1-r(\eta))$ and  
equations (20), (21) are satisfied, we have 
$\E(\alpha) = \left< \alpha', \alpha'\right> = - b^{2} + c^{2}$, where
$$-b^{2} = \lpr{\beta'}{\beta'} = \frac{r}{r-1}( (r')^{2} -1) \; \; \mbox{ and } \; \; 
c^{2} = \lpr{\gamma'}{\gamma'} = \frac{L^{2}}{r^{2}}.$$ 

Now, substituting $b^{2} = c^{2} - \E(\alpha)$ in the last equation above it follows that
$$\frac{L^{2}}{r^{2}} - \E(\alpha) = \frac{r}{r - 1}(1 - (r')^{2}),$$
which implies 
\begin{equation}\label{6}
\frac{dr}{d\eta} = \pm \sqrt{1 + \left(\E(\alpha) - \frac{L^{2}}{r^{2}}\right) \frac{r - 1}{r}},
\end{equation}
and so the equation (\ref{4}) follows easily.  From equation (\ref{6}) one now has
$$(r')^{2} = \frac{1}{r^{3}} \left[r^{3} + \E(\alpha) r^{2}(r - 1) - L^{2} (r - 1) \right],$$ 
which implies that equation (\ref{5}) holds. 
\end{proof}

\subsection{Consequences of equation (23) and (\ref{5})}
Here we obtain some results to geodesic line $\alpha(\eta)$ with the condition $\varphi = \frac{\pi}{2}$ and $r(\eta) \in ]0,1[$.

\begin{corol}
Except for periodic orbits, the geodesic line $\alpha(\eta)$ is asymptotically Newtonian if, and only if, 
$\E(\alpha) = -1$. 

Moreover, when $\varphi = \pi/2$  the linking equation can be written as
\begin{equation}\label{10}
r'' = \frac{\E(\alpha)}{2r^{2}} + \frac{L^{2}}{2r^{4}}(2r - 3),
\end{equation} 

and the equation for $\hat{\gamma}(\eta)$ in Proposition \ref{12} takes the form
\begin{equation}\label{11}
\hat{\gamma}''(\eta) + \frac{1}{2r^{3}}\left(-\E(\alpha) + \frac{3L^{2}}{r^{2}}\right) \hat{\gamma}(\eta) = 0. 
\end{equation} 

Furthermore, a necessary and sufficient condition to extend a solution of equation (\ref{11}) to a geodesic line of 
$(M,\mathbf{g})$ is that the function $r(\eta)$ satisfies the equation (\ref{10}). 
\end{corol}

\begin{proof}
Formula (\ref{10}) follows directly from deriving  equation (\ref{5}).  Now, formula (\ref{11}) 
follows from considering the curve in rectangular coordinates $\hat \gamma = (x,y,z)$ as in Proposition \ref{12} 
to get through of equation (\ref{60}) that  $$ac(\eta) = \frac{1}{2r^3}[ \E(\alpha) - \frac{3L^2}{r^2}].$$ 
\end{proof} 

Next, we will give an example of a lightlike plane section of the lightlike cone in $\mathbb{L}(\partial_{3})$, 
which can not be extended to a geodesic line of the Schwarzschild spacetime. In order to do that, we establish 
the following proposition.

\begin{prop}\label{01}
The parametric curve $\hat{\gamma}(\eta) = (0, x(\eta), y(\eta), 0)$ satisfies the condition of Proposition \ref{12} 
if, and only if, 
$x y' - x' y = L \in \real.$ Moreover, 
the lifting curve $\gamma = (\pm \sqrt{x^{2} + y^{2}}, x, y, 0)$ maps $I$ into $\C$.
\end{prop} 
 
Then we affirm the next. 
\begin{example}
The lightlike plane section of the lightcone of \ $\mathbb{L}(\partial_{3})$ given by 
\begin{equation}\label{14}
\gamma(x) = \left(\frac{x^{2} + a^{2}}{2a}, x, \frac{x^{2} - a^{2}}{2a}, 0\right), \; \; \mbox{ for } \; \; a > 0,
\end{equation}
cannot be extended to a geodesic line of $(M,\mathbf{g})$. 

\vspace{0.2cm}
To check that we start taking  the branch $\gamma(\eta) = (r(\eta), \sqrt{2 a r(\eta) - a^{2}}, r(\eta) - a, 0)$ 
and assume that the parameter $\eta$ is such that the geodesic equations hold with $L$ given by 
$x y' - x' y = L \in \real$, where $y = r - a$. According Proposition \ref{01}, one has $x = \sqrt{2ar - a^2}$ and $y = r - a$. 
In this case, $\varphi = \pi/2$. Now, one calculates $x'$, $y'$ and after a directed computation we get \ 
$x y' -y x' = \displaystyle \frac{ar r'}{\sqrt{2ar - a^2}} =L$.  From the latter relation it follows that 
$r' = \displaystyle \frac{L}{a r}\sqrt{2 a r - a^{2}}$ and hence  $$r'' = \frac{L^{2}}{a r^{3}}(a - r).$$
Replacing $r''$ in the equation (\ref{10}) it follows that $\E(\alpha)$ cannot be a constant function, and so the curve (\ref{14}) 
cannot be extended to a geodesic line of $(M,\mathbf{g})$. 
\end{example}

\vspace{0.2cm}

\end{document}